\documentclass[12pt]{amsart}
\usepackage{amsmath,amsthm,amscd,amssymb,bbm,graphicx}
\usepackage{floatflt,setspace, wrapfig}
\usepackage{graphics,multicol}
\usepackage{mathrsfs} 
\usepackage{epstopdf}
\usepackage{geometry}

\newtheorem{thrm}{Theorem}

\newtheorem{cor}{Corollary}
\newtheorem{lem}{Lemma}
\newtheorem{remark}{Remark}

\newcommand{\ind}{\mathbbm{1}}

\newcommand{\diam}{\mbox{diam }}

\newcommand{\const}{\mbox{const.}}
\newcommand{\id}{\mbox{id}}
\newcommand{\eps}{\varepsilon}

\providecommand{\abs}[1]{\lvert \, #1 \, \rvert}
\providecommand{\Abs}[1]{\biggl\lvert \, #1 \, \biggr \rvert}

\providecommand{\floor}[1]{\left\lfloor \, #1 \, \right\rfloor}

\def\T{\hat{T}}
\def\ox{\vec{x}}
\def\a{\alpha}
\def\ox{\overline{x}}
\def\g {\gamma}
\def\Le{\text{Leb}}
\newcommand{\bea}[1]{\begin{eqnarray}\label{#1}}
\newcommand{\eea}{\end{eqnarray}}

\title[Limiting law for arbitrary sets]{Limiting entry times distribution for arbitrary null sets}

\begin{document}

\author{Nicolai Haydn}
\thanks{N Haydn, Mathematics Department, USC,
Los Angeles, 90089-2532. E-mail: {\tt {nhaydn@math.usc.edu}}}
 \author{Sandro Vaienti}
 \thanks{S Vaienti, Aix Marseille Universit\'e, Universit\'e de Toulon, CNRS, CPT, UMR 7332, Marseille, France.
 E-mail: {\tt {vaienti@cpt.univ-mrs.fr}}}

\date{\today}



\begin{abstract}
We describe an approach that allows us to deduce the limiting
return times distribution for arbitrary sets to be compound Poisson distributed.
We establish a relation between the limiting return times distribution and
the probability of the cluster sizes, where clusters consist of the portion of
points that have finite return times in the limit where random return times
go to infinity. In the special case of periodic points we recover the known
P\'olya-Aeppli distribution which is associated with geometrically distributed
cluster sizes. We apply this method to several examples the most
important of which is synchronisation of coupled  map lattices.
For the invariant absolutely continuous measure we establish that the returns to the
 diagonal is compound Poisson distributed where the coefficients are
 given  by certain integrals along the diagonal.
\end{abstract}

\maketitle
\tableofcontents

\section{Introduction}

Return times statistics have recently been studied quite extensively. For
equilibrium states for H\"older continuous potentials on Axiom A systems
in particular, Pitskel~\cite{Pit91} showed that generic points  have in the limit
Poisson distributed return times if one uses cylinder neighbourhoods.
 In the same paper he also showed that this result applies only almost
 surely and shows that at periodic points the return times distribution
 has a point mass at the origin which corresponds to the periodicity
 of the point. It became clear later that in fact for every non-periodic
 point the return times are in the limit Poisson distributed while
 for periodic points the distribution is P\'olya-Aeppli which
 is a Poisson distribution compounded with a geometric distribution of clusters,
 where the parameter for the geometric distribution is the value
 given by Pitskel. For $\phi$-mixing systems in a symbolic setting, this dichotomy follows
 from~\cite{Aba06}. For more general classes of dynamical systems with various kind of mixing properties,
 we showed in our paper \cite{HV09} that  limiting return times distributions at periodic points
 are almost everywhere compound Poissonian;
moreover we derived error terms for the
convergence to the limiting distribution in many other settings. The paper~\cite{KR14} showed that all
$\psi$-mixing shifts the limiting distributions of the numbers of multiple recurrencies to shrinking
cylindrical neighborhoods of all points are close either to Poisson or to compound Poisson distributions.
In the classical setting this dichotomy was shown in~\cite{HP14} using the Chen-Stein method
for $\phi$-mixing measures, where
for cylinder sets the limiting distribution was found to be Poisson at all non-periodic points.
Extension to non-uniformly hyperbolic dynamical systems are provided in \cite{FFT13},
which establishes and discusses the connection between the laws of Return Times Statistics and
Extreme Value Laws  (see also the book~\cite{B16} for a panorama and an account on extreme value theory
and point processes applied to dynamical systems). For planar dispersing billiards the return times distribution
 is, in the limit, Poisson for metric balls almost everywhere w.r.t.\ the SRB measure:
 this has been proved in~\cite{FHN14}. Convergence in distribution for the rescaled
 return times in planar billiard has been shown in~\cite{PS10} where the same authors proved that
  the distribution of the number of visits to a ball with vanishing radius  converges to a Poisson distribution
   for some nonuniformly hyperbolic invertible dynamical systems which are modeled by a
   Gibbs-Markov-Young tower~\cite{PS16}. Similarly~\cite{CC13} established Poisson approximation
   for metric balls for  systems modelled by a Young tower whose return-time function has a
   exponential tail and with one-dimensional unstable manifolds, which included the H\'enon attractor.
   For polynomially decaying correlations this was done in~\cite{HW16} where also the restriction on the
   dimension of the unstable manifold was dropped.
    In a more geometric setting the limiting distribution for shrinking balls was shown in~\cite{HY17}.
     Spatio-temporal Poisson processes obtained from recording not only the successive times of visits to a set,
     but also the positions, have been recently studied in~\cite{PS18}.
     Another kind of extension has been proposed in~\cite{FFM17}, which  studied marked point processes
     associated to extremal observations  corresponding to exceedances of high thresholds.
     Finally distributions of return to different sets of cylinders have been recently  considered in~\cite{KY18}.

In the current paper we look at a more general
 setting which allows us to find the limiting return times distribution
 to an arbitrary zero measure set $\Gamma$ by looking at the
 return times distribution of a neighbourhood $B_\rho(\Gamma)$
 on a time scale suggested by Kac's lemma.
 For the approximating sets we then show that the
 return times are close to compound binomially distributions
 (Theorem~\ref{helper.theorem}), which in the limit
 converges to a compound Poissonian.
 We show this in a geometric setup that requires that the correlation
 functions decay at least polynomially. The slowest rate required
 depends on the regularity of the invariant measure.

 We then apply this result to some examples which include
 the standard periodic point setting. It also allows us to look
 at coupled map lattices, where the diagonal set is invariant.
 The return times statistics then expresses the degree to which
 neighbouring points are synchronised.

 In the next section we describe the systems we want to
 consider and state the main result, Theorem~\ref{thm1}.
 In Section~\ref{return.times} we connect the distribution of
 the return times functions to the probabilities of the cluster
 sizes which are the parameters that describe the limiting
 distribution.
 Section~\ref{binomial} consists of a very general approximation
 theorem that allows us to measure how close a return times
 distribution is to being compound binomial.
 Section~\ref{proof.theorem1} contains the proof of the main result.
 Section~\ref{examples} has some examples including the pathological Smith
 example and standard periodic points.
 Section~\ref{synchronisation} deals with coupled map lattices, where
 the maps that are coupled are expanding interval maps.
 There we show that for the absolutely continuous invariant measure
  the parameters for the compound Poisson limiting distribution
 are given by integrals along the diagonal. In particular one sees
 that is this case the parameters are in general not geometrical.

  \section{Compound Poisson Distribution}

  An integer valued random variable $W$ is compound Poisson distributed
  if there are i.i.d.\ integer valued random variables $X_j\ge 1$, $j=1,2,\dots$,
  and an independent Poisson distributed random variable $P$ so
  that $W=\sum_{j=1}^PX_j$. The Poisson distribution $P$ describes the
  distribution of clusters whose sizes are described by the random variables $X_j$
  whose probability densities are given by
values $\lambda_\ell=\mathbb{P}(X_j=\ell)$, $\ell=1,2,\dots$.

We say a probability measure $\tilde\nu$ on $\mathbb{N}_0$ is compound
Poisson distributed with parameters $s\lambda_\ell$, $\ell=1,2,\dots$,
if its generating function $\varphi_{\tilde\nu}$ is given
$\varphi_{\tilde\nu}(z)=\exp\int_0^\infty(z^x-1)\,d\rho(x)$,
where $\rho$ is the measure on
 $\mathbb{N}$ defined by $\rho =\sum_\ell s\lambda_\ell\delta_\ell$,
with $\delta_\ell$  being the point mass at $\ell$.
If we put $L=\sum_\ell s\lambda_\ell$
then $L^{-1}\rho$ is a probability measure and the random variable $W=\sum_{j=1}^PX_j$
is compound Poisson distributed, where $P$ is Poisson distributed with parameter $L$ and
$X_j,  j=1,2,\dots$, are i.i.d.\ random variables with distribution
$\mathbb{P}(X_j=\ell)=\lambda_\ell=L^{-1}s\lambda_\ell$,
$\ell=1,2,\dots$. In our case $L=s$ and moreover $\mathbb{E}(W)=s\mathbb{E}(X_0)$.
 In the special case $X_1=1$ and  $\lambda_\ell=0\forall \ell\ge2$ we recover the
 Poisson distribution $W=P$.
 The generating function is then $\varphi_{W}(z)=\exp(-s(1-\varphi_X(z)))$, where
 $\varphi_X(z)=\sum_{\ell=1}^\infty z^\ell\lambda_\ell$ is the generating function of $X_j$.

 An important non-trivial compound Poisson distribution is the P\'olya-Aeppli
 distribution which happens when  the $X_j$ are geometrically distributed,
 that is $\lambda_\ell=\mathbb{P}(X_\ell)=(1-p)p^{\ell-1}$ for $\ell=1,2,\dots$, for some $p\in(0,1)$.
 In this case
 $$
 \mathbb{P}(W=k)=e^{-s}
\sum_{j=1}^kp^{k-j}(1-p)^j\frac{s^j}{j!}\binom{k-1}{j-1}
$$
and in particular $\mathbb{P}(W=0)=e^{-s}$.
In the case of $p=0$ this reverts back to the straight Poisson distribution.

   In our context when we count limiting
  returns to small sets, the Poisson distribution gives the distribution
  of clusters which for sets with small measure happens on a large
  timescale as suggested by Kac's formula. The number of returns
  in each cluster is given by the i.i.d.\ random variables $X_j$.
  These returns are on a fixed timescale and nearly independent of
  the size of the return set as its measure is shrunk to zero.

\section{Assumptions and main results} \label{assumptions}

\subsection{The counting function}
Let $M$ be a manifold and $T:M\to M$ a $C^2$ local diffeomorphism with the properties described below in
the assumptions. We envisage both cases of  global invertible maps eventually with singularities  and maps which are locally injective on a suitable partition of $M.$  Let $\mu$ be a $T$-invariant Borel probability measure on $M$.

For a subset $U\subset M$, $\mu(U)>0,$  we define the counting function
\vspace{-0.2cm}
\begin{equation*}
\xi^{t}_{U}(x)=\sum_{n=0}^{\floor{t/\mu(U)}} {\ind_{U}} \circ T^n(x).
\end{equation*}
which tracks the number of visits a trajectory of the point $x \in M$ makes to the set $U$
on an orbit segment of length $N=\floor{t/\mu(U )}$, where $t$ is a positive
parameter.
(We often omit the sub- and superscripts and simply use $\xi(x)$.)

\subsection{The hyperbolic structure and cylinder sets}
Let $\Gamma^u$ be a collection of unstable leaves $\gamma^u$
and  $\Gamma^s$ a collection of stable leaves $\gamma^s$. We assume
that $\gamma^u\cap\gamma^s$ consists of a single point for all $(\gamma^u,\gamma^s)\in\Gamma^u\times\Gamma^s$.
The map $T$ contracts along the stable leaves (need not to be uniform) and similarly $T^{-1}$
contracts along the unstable leaves.

For an unstable leaf $\gamma^u$ denote by $\mu_{\gamma^u}$ the
disintegration of $\mu$ to the $\gamma^u$. We assume that $\mu$
has a product like decomposition
$d\mu=d\mu_{\gamma^u}d\upsilon(\gamma^u)$,
where $\upsilon$ is a transversal measure. That is, if $f$ is a function on $M$ then
$$
\int f(x)\,d\mu(x)
=\int_{\Gamma^u} \int_{\gamma^u}f(x)\,d\mu_{\gamma^u}(x)\,d\upsilon(\gamma^u)
$$

If $\gamma^u, \hat\gamma^u\in\Gamma^u$ are two unstable leaves then the holonomy map $\Theta:\gamma^u\to \hat\gamma^u$ is defined by $\Theta(x)=\hat\gamma^u\cap\gamma^s(x)$ for $x\in\gamma^u$,  where $\gamma^u(x)$ be the local unstable leaf through $x$.

Let us denote by
$J_n=\frac{dT^n\mu_{\gamma^u}}{d\mu_{\gamma^u}}$
the Jacobian of the map $T^n$ with respect to the measure $\mu$
in the unstable direction.

Let $\gamma^u$ be a local unstable leaf.
Assume there exists $R>0$ and  for every $n\in\mathbb{N}$ finitely many
 $y_k\in T^n\gamma^u$ so that
$T^n\gamma^u\subset\bigcup_k B_{R,\gamma^u}(y_k)$,
where $B_{R,\gamma^u}(y)$ is the embedded $R$-disk centered at $y$ in the unstable leaf
 $\gamma^u$.
Denote by $\zeta_{\varphi,k}=\varphi(B_{R,\gamma^u}(y_k))$ where $\varphi\in \mathscr{I}_n$
 and $\mathscr{I}_n$ denotes the  inverse branches of $T^n$.
 We call $\zeta$ an $n$-cylinder. In the case of piecewise expanding endomorphisms in any dimension, we will define an $n$-cylinder $\zeta_n$ as an element of the join partition $\mathcal{A}^n:= \bigvee_{j=0}^{n-1}T^{-j}\mathcal{A},$ where $\mathcal{A}$ is the initial partition into subsets of monotonicity for the map $T.$

\subsection{Assumptions}
We shall make two sets of assumptions, the first two will be on the map
and the properties of the invariant measure per se, while Assumptions~(IV),
(V) and~(VI) will involve the approximating sets of $\Gamma$. The sets $\mathcal{G}_n$
account for possible discontinuity sets of the map where the derivative might become
singular in a controlled way.\\
(I) {\em Overlaps of cylinders:} There exists a constant $L$ so that the number of overlaps
$N_{\varphi,k}=|\{\zeta_{\varphi',k'}: \zeta_{\varphi,k}\cap\zeta_{\varphi',k'}\not=\varnothing,
\varphi'\in\mathscr{I}_n\}|$
is bounded by $L$ for all $\varphi\in \mathscr{I}_n$ and for all $k$ and $n$.
This follows from the fact that
$N_{\varphi,k}$ equals $|\{k': B_{R,\gamma^u}(y_k)\cap B_{R,\gamma^u}(y_{k'})\not=\varnothing\}|$
which is uniformly bounded by some constant $L$. For endomorphisms the analogous  requirement will  be that there exists $\iota>0$ such that for any $n$ and any $n$-cylinder $\zeta_n\in \mathcal{A}^n$ we have $\mu(T^n\zeta_n)>\iota.$  \\
(II) {\em Decay of correlations:} There exists a decay function $\mathcal{C}(k)$ so that
$$
\left|\int_MG(H\circ T^k)\,d\mu-\mu(G)\mu(H)\right|
\le \mathcal{C}(k)\|G\|_{Lip}\|H\|_\infty\qquad\forall k\in\mathbb{N},
$$
for functions $H$ which are constant on local stable leaves $\gamma^s$ of $T$.
The functions $G:M\to\mathbb{R}$ are Lipschitz continuous w.r.t.\ the given metric on $M$.
\\
(III)   Assume there are sets $\mathcal{G}_n$ so that \\
\indent (i) {\em Non-uniform setsize:} $\mu(\mathcal{G}_n^c)=\mathcal{O}(n^{-q})$
for some positive $q$.\\
\indent (ii) {\em Distortion:} $\frac{J_n(x)}{J_n(y)}=\mathcal{O}(\omega(n))$ for all $x,y\in\zeta$,
$\zeta\cap\mathcal{G}_n^c=\varnothing$ for $n\in \mathcal{N}$, where
$\zeta$ are $n$-cylinders in unstable leaves $\gamma^u$ and $\omega(n)$ is a non-decreasing
sequence.\\
\indent (iii) {\em Contraction:} There exists a $\kappa > 1$, so that
$\diam\zeta\le n^{-\kappa}$ for all $n$-cylinders $\zeta$ for which $\zeta\cap\mathcal{G}_n^c=\varnothing$
 and all $n$.\\

 \vspace{3mm}

 \noindent
 Now assume $\Gamma\subset M$ is a zero measure set that is approximated by
 sets $U=B_\rho(\Gamma)$. We then make the following assumptions:\\
(IV) {\em Dimension:} There exist $0<d_0<d_1$ such that
$\rho^{d_0}\ge\mu(B_\rho(\Gamma))\ge\rho^{d_1}$.\\
(V) {\em Unstable dimension:} There exists a $u_0$ so that
$\mu_{\gamma^u}(B_\rho(\Gamma))\le C_1\rho^{u_0}$ for all $\rho>0$ small enough
and for almost all $x\in \gamma^u$, every unstable leaf $\gamma^u$.\\
(VI) {\em Annulus type  condition:} Assume that for some $\eta, \beta>0$:
$$
\frac{\mu(B_{\rho+r}(\Gamma)\setminus B_{\rho-r}(\Gamma))}{\mu(B_\rho(\Gamma))}
= \mathcal{O}(r^\eta\rho^{-\beta})
$$
for every $r< \rho_0$ for some $\rho_0<\rho$ (see remark below).

Here and in the following we use the notation $x_n\lesssim y_n$ for $n=1,2, \dots$, to mean that
there exists a constant $C$ so that $x_n<Cy_n$ for all $n$.
 As before let
$T:\Omega\circlearrowleft$ and $\mu$ a $T$-invariant probability measure
on $\Omega$. For a subset $U\subset\Omega$ we put $I_i=\ind_U\circ T^i$ and define
$$
Z^L=Z^L_U=\sum_{i=0}^{2L}I_i
$$
where $L$ is a (large) positive integer.

Let us now formulate our main result.

\begin{thrm}\label{thm1}
Assume that the map $T: M\to M$ satisfies the assumptions (I)--(VI) where $\mathcal{C}(k)$ decays at
least polynomially with power $p>\frac{\frac\beta\eta+d_1}{d_0}$. Moreover we assume that
$d_0>\max\{\frac{d_1}{q-1},\frac\beta{\kappa\eta-1}\}$ and $\kappa u_0>1$
Assume $\omega(j)\lesssim j^{\kappa'}$ for some $\kappa'\in[0,\kappa u_0-1)$.
Let $\Gamma\subset M$ be a zero measure set, let $t>0$ and put
 $\lambda_\ell=\lim_{K\to\infty}\lambda_\ell(K)$,
where
$$
\lambda_\ell(K)=\mathbb{P}(Z^K_{B_\rho(\Gamma)}=\ell)/
\mathbb{P}(Z^K_{B_\rho(\Gamma)}\ge1).
$$

Then
$$
\mathbb{P}(\xi_{B_\rho(\Gamma)}^t=k)\longrightarrow\nu(\{k\})
$$
as $\rho\to0$, where $\nu$ is the compound Poisson distribution for the parameters
$s\lambda_\ell$, where $s=\alpha_1t$.
\end{thrm}

\begin{remark}\label{PR}
 In the classical case when the limiting set consists of a single point, namely $\Gamma=\{x\}$,
then we recover the known results which are the two cases when  $x$ is a non-periodic point and
when $x$ is a periodic point. If $x$ is a non-periodic point then $\lambda_1=1$ and $\lambda_\ell=0$ for $\ell\ge2$
which implies that the limiting distribution is Poissonian. Previously this was shown in~\cite{CC13}
for exponentially decaying correlations and in~\cite{HW16} for polynomially decaying correlations.
Another more general version is given in~\cite{HY17}. In the case when $x$ is periodic
we obtain that $\lambda_\ell=(1-p)p^{\ell-1}$ for all $\ell=1,2,\dots$, and where
$p$ is given by the limit $\lim_{\rho\to0} \frac{\mu(B_\rho(x)\cap T^{-m}B_\rho(x))}{\mu(B_\rho(x))}$ if the limit exists and
where $m$ is the minimal period of $x$. The limiting distribution in this case is P\'olya-Aeppli.
 Pitskel~\cite{Pit91} obtained this value for equilibrium states for
Axiom A systems and a more general description is found in~\cite{HV09}.
See also section~\ref{section.periodic.points}.
\end{remark}

The proof of  Theorem~\ref{thm1} is given in  Section~\ref{proof.theorem1}.
In the following section we will express the parameters $\lambda_\ell$ in terms of the limiting
return times distribution.


\section{Return times}\label{return.times}

In this section we want to relate the parameters $\lambda_k$ which determine the
limiting probability of a $k$-cluster to occur to the return times distribution.
To account for a more general setting, let $T:\Omega\circlearrowleft$ be a measurable
map on a space $\Omega$.
For a subset $U\subset\Omega$ we define the first entry/return time
 $\tau_U$ by $\tau_U(x)=\min\{j\ge1: T^j\in U\}$. Similarly we get higher order
 returns by defining recursively $\tau_U^\ell(x)=\tau_U^{\ell-1}+\tau_U(T^{\tau_U^{\ell-1}}(x))$
 with $\tau_U^1=\tau_U$. We also write $\tau_U^0=0$ on $U$.

Let $U_n\subset \Omega$, $n=1,2,\dots$, be a nested sequence of sets and put
$\Lambda=\bigcap_nU_n$.
For  $K$ be a large number which later will go to infinity and
assume the limits
  $\hat\alpha_\ell(K)=\lim_{n\to\infty}\mu_{U_n}(\tau_{U_n}^{\ell-1}\le K)$
  exist for $K$ large enough.
  Since $\{\tau_{U_n}^{\ell+1}\le K\}\subset \{\tau_{U_n}^{\ell}\le K\}$ we get that
  $\hat\alpha_\ell(K)\ge\hat\alpha_{\ell+1}(K)$ for all $\ell$ and in particular $\hat\alpha_1(K)=1$.
  By monotonicity the limits $\hat\alpha_\ell=\lim_{K\to\infty}\hat\alpha_\ell(K)$
  exist and satisfy $\hat\alpha_1=1$ and $\hat\alpha_\ell\ge\hat\alpha_{\ell+1}$ $\forall \ell$.

  Now assume that moreover the limits $p_i^\ell=\lim_{n\to\infty}\mu_{U_n}(\tau_{U_n}^{\ell-1}=i)$
  of the conditional size of the level sets of the $\ell$th return time $\tau_{U_n}^\ell$
  exist for $i=0,1,2,\dots$ (clearly $p_i^\ell=0$ for $i\le\ell-2$). Then we can formulate the following relation.

  \begin{lem} For $\ell=2,3,\dots$:
  $$
  \hat\alpha_\ell=\sum_ip_i^\ell.
  $$
  \end{lem}

  \begin{proof}  Let $\varepsilon>0$, then there exists $K_1$ so that
  $|\hat\alpha_\ell-\hat\alpha_\ell(K)|<\varepsilon$ for all $K\ge K_1$.
 Let $K\ge K_1$, then for all small enough $U$ one has
 $|\hat\alpha_\ell(K)-\mu_U(\tau_U^{\ell-1}\le K)|<\varepsilon$.
 Thus $|\hat\alpha_\ell-\mu_U(\tau_U^{\ell-1}\le K)|<2\varepsilon$.
 There exists $K_2$ so that $\sum_{i=K+1}^\infty p_i^{\ell-1}<\varepsilon$ for all $K\ge K_2$.
 If we let $K\ge K_0=K_1\vee K_2$ then for all small enough $U$ one has
 $|p_i^\ell-\mu_U(\tau_U^{\ell-1}=i)|<\varepsilon/K$.

Consequently
$$
\hat\alpha_\ell=\sum_{i=1}^K\mu_U(\tau_U^{\ell-1}=i)+\mathcal{O}(2\varepsilon)
=\sum_{i=1}^Kp_i^{\ell-1}+\mathcal{O}(3\varepsilon)
=\sum_{i=1}^\infty p_i^{\ell-1}+\mathcal{O}(4\varepsilon).
$$
Now let $\varepsilon$ go to zero.
\end{proof}

Now put $\alpha_\ell=\lim_{K\to\infty}\alpha_\ell(K)$,
where $\alpha_\ell(K)=\lim_{n\to\infty}\mu_{U_n}(\tau_{U_n}^{\ell-1}\le K<\tau_{U_n}^\ell)$
for $\ell=1,2,\dots$.
Since $\{\tau_{U_n}^\ell\le K\}\subset\{\tau_{U_n}^{\ell-1}\le K\}$
we get
$\{\tau_{U_n}^{\ell-1}\le K<\tau_{U_n}^\ell\}=\{\tau_{U_n}^{\ell-1}\le K\}\setminus\{\tau_{U_n}^{\ell}\le K\}$.
Therefore $\alpha_\ell=\hat\alpha_\ell-\hat\alpha_{\ell+1}$ which in particular implies the
existence of the limits $\alpha_\ell$.  Also, by the previous lemma
$$
\alpha_\ell=\sum_i(p_i^{\ell-1}-p_i^\ell)
$$
for $\ell=2,3,\dots$. In the special case $\ell=1$ we get in particular
$\alpha_1=\lim_{K\to\infty}\lim_{n\to\infty}\mu_{U_n}(K<\tau_{U_n})$.
Since $p_0^1=1$ and $p^1_i=0\;\forall\; i\ge1$ we get
$\alpha_1=1-\sum_ip^2_i$.

Dropping the index $n$, let $I_0=\ind_U$ the characteristic function and put $I_i=I_0\circ T^j$, then
we can define the random variable $Z=\sum_{j=0}^KI_j$ and obtain that
 $\lim_U\mathbb{E}(\ind_{Z=\ell}|I_0)=\lim_U\mu_U(Z=\ell)=\alpha_\ell(K)$.

Now put
$$
\lambda _k(L,U)=\mathbb{P}(Z^L=k|Z^L>0)=\frac{\mathbb{P}(Z^L=k)}{\mathbb{P}(Z^L>0)}.
$$
For a sequence of sets $U_n$ for which $\mu(U_n)\to0$ as $n\to\infty$ we
put $\lambda_k(L)=\lim_{n\to\infty}\lambda_k(L,U_n)$. Evidently
$\lambda_k(L,U)\le\lambda_k(L',U)$ if $L\le L'$ and consequently also
$\lambda_k(L)\le\lambda_k(L')$. As a result the limit $\lambda_k=\lim_{L\to\infty}\lambda_k(L)$
always exists.

Let us also define $Z^{L,+}=Z^{L,+}_U=\sum_{i=L}^{2L}I_i$ and similarly
 $Z^{L,-}=Z^{L,-}_U=\sum_{i=0}^{L-1}I_i$. Evidently $Z^L=Z^{L,-}+Z^{L,+}$
 and moreover
 $$
 \alpha_k=\lim_{L\to\infty}\lim_{n\to\infty}\mathbb{P}(Z_U^{L,+}=k|I_L=1)
 $$
 which by invariance is equal to
 $ \alpha_k=\lim_{L\to\infty}\lim_{n\to\infty}\mathbb{P}(Z_U^{L}=k|I_0=1)$.
\noindent Let us notice that $\alpha_1$ is commonly called the {\em extremal index}.
Let us define $W^L=\sum_{i=0}^LI_i$.
Then $\alpha_k=\lim_{L\to\infty}\lim_{n\to\infty}\mathbb{P}(W^L=k| I_0=1)$.

\begin{lem} \label{tail.lemma}
Assume that for all $L$ large enough the limits
 $\hat\alpha_k(L)=\lim_{n\to\infty}\hat\alpha_k(L,U_n)$
 exist along a (nested) sequence of sets $U_n$, $\mu(U_n)\to0$ as $n\to\infty$.
 Assume $\sum_{k=1}^\infty k\hat\alpha_k<\infty$ where $\hat\alpha_k=\lim_{L\to\infty}\hat\alpha_k(L)$.

 Then for every $\eta>0$ there exists an $L_0$ so that for all $L'>L\ge L_0$:
 $$
\mathbb{P}(W^{L'-L}\circ T^L>0, I_0=1)\le \eta\mu(U_n)
$$
and
$$
\mathbb{P}(W^L>0, I_{L'}=1)\le \eta\mu(U_n)
$$
for all $n$ large enough (depending on $L,L'$).
\end{lem}

\begin{proof}
(I) To prove the first estimate, let $\varepsilon>0$ and $k\ge1$.
Let $k_0$ be so that $\sum_{k=k_0}^\infty\hat\alpha_k<\varepsilon$
and then $L_0$ large enough so that $\hat\alpha_k-\hat\alpha_k(L)<\varepsilon/k_0$ for
all $L\ge L_0$. Then for all sufficiently large $n$  one has
$\left|\hat\alpha_k(L)-\hat\alpha_k(L,U_n)\right|<\varepsilon/k_0$ for all $k\le k_0$.
Also, for $n$ large enough we can achieve that
$\sum_{k=k_0}^\infty\hat\alpha_k(L,U_n)=\sum_{k=k_0}^L\hat\alpha_k(L,U_n)\le2 \varepsilon$.
From now on $U=U_n$.

Note that
$\hat\alpha_k(L,U)=\mathbb{P}(W^L\ge k|I_0=1)$
and
$$
U\cap\{W^L\ge k\}\subset U\cap \{W^{L'}\ge k\},
$$
where $U=\{I_0=1\}$. Consequently
$$
U\cap\{W^{L'}\ge k\}\setminus\{W^L\ge k\}
=U\cap T^{-L}\{W^{L'-L}>0\}\cap\{W^{L'}\ge k\}
$$
and therefore
$$
\mathbb{P}(I_0=1,W^{L'-L}\circ T^L>0,W^{L'}\ge k)
=\mu(U)(\hat\alpha_k(L',U)-\hat\alpha_k(L,U)).
$$
Hence
$$
\mathbb{P}(W^{L'-L}\circ T^L>0,I_0=1)
\le\mu(U)\sum_{k=1}^\infty(\hat\alpha_k(L',U)-\hat\alpha_k(L,U))
<5\varepsilon\mu(U)
$$
since $\sum_{k=k_0}^\infty\hat\alpha_k(L,U)\le3 \varepsilon$.
The first inequality of the lemma now follows if $\varepsilon=\eta/5$.
\vspace{3mm}

(II) To prove the second bound let $\varepsilon>0$ and $k\ge1$.
Let $k_0$ be so that $\sum_{k=k_0}^\infty k\hat\alpha_k<\varepsilon$
and then $L_0$ large enough so that $\hat\alpha_k-\hat\alpha_k(L)<\varepsilon/k_0$ for
all $L\ge L_0$. Then for all sufficiently large $n$  one has
$\left|\hat\alpha_k(L)-\hat\alpha_k(L,U_n)\right|<\varepsilon/k_0$ for all $k\le k_0$.
Moreover for $n$ large enough we also obtain
$\sum_{k=k_0}^\infty\hat\alpha_k(L',U_n)=\sum_{k=k_0}^L\hat\alpha_k(L',U_n)<2\varepsilon$.
Let $U=U_n$ and notice that
$$
\mathbb{P}=(W^L>0,I_{L'}=1)
=\sum_{k=1}^\infty\mathbb{E}(\ind_{W^L=k}I_{L'})
=\sum_k\frac1k\mathbb{E}(\ind_{W^L=k}W^LI_{L'})
=\sum_k\frac1k\sum_{i=0}^L\mathbb{E}(\ind_{W^L=k}I_iI_{L'})
$$
and
$$
\bigcup_{i=0}^L\{W^L=k,I_i=1,I_{L'}=1\}
=\bigcup_{i=0}^L\bigcup_{\vec{i}\in J^k}\left(C_{\vec{i}}\cap\{I_i=I_{L'}=1\}\right),
$$
where
$$
J^k=\left\{\vec{i}=(i_1,i_2,\dots,i_k): 0\le i_1<i_2<\cdots<i_k\le l\right\}
$$
and
$$
C_{\vec{i}}=\left\{I_{i_j}=1\forall i=j,\dots,k, \;\;I_a=0 \forall a\in[0,L]\setminus\{i_j:j\}\right\}.
$$
Then
\begin{eqnarray*}
\bigcup_{i=0}^L\{W^L=k\}\cap\{I_i=I_{L'}=1\}
&=&\bigcup_{j=1}^k\bigcup_{\vec{i}\in J^k}\left(C_{\vec{i}}\cap\{I_i=I_{L'}=1\}\right)\\
&=&\bigcup_{j=1}^k\bigcup_{p=0}^LT^{-p}\!\left(\bigcup_{\vec{i}\in J^k_p(j)}\left(C_{\vec{i}}\cap\{I_0=I_{L'-p}=1\}\right)\right),
\end{eqnarray*}
where
$$
J^k_p(j)=\left\{\vec{i}=(i_1,\dots,i_k): -p\le i_1<\cdots<i_k\le L-p, \;\;i_j=p,\;\;I_a=0 \forall a\in[-p,L-p]\setminus\{i_j:j\}\right\}
$$
(put $J^k_p(j)=\varnothing$ if either $p<j$ or $p>L-j$).
Consequently
$$
\{W^L>0, I_{L'}=1\}
=\bigcup_{p=0}^LT^{-p}\!\left(\bigcup_{k=1}^\infty\bigcup_{j=1}^k\bigcup_{\vec{i}\in J^k_p(j)}
\left(C_{\vec{i}}\cap\{I_0=I_{L'-p}=1\}\right)\right)
$$
where the triple union inside the brackets is a disjoint union.
Thus
\begin{eqnarray*}
\mathbb{P}(W^L>0, I_{L'}=1)
&\le&\sum_{p=0}^L\mathbb{E}(I_0I_{L'-p})\\
&=&\mathbb{E}(W^L\circ T^{L'-L}I_0)\\
&\le&k_0\mathbb{P}(W^L\circ T^{L'-L}>0,I_0=1) +\sum_{k=k_0}^\infty k\mathbb{P}(W^L\circ T^{L'-L}=k,I_0=1)\\
&\le&k_05\varepsilon\mu(U)+\sum_{k=k_0}^\infty k\hat\alpha_k(L',U)\\
&\le&7\varepsilon \mu(U)
\end{eqnarray*}
where we used the estimate from Part~(I). Now put $\varepsilon=\eta/7$.
\end{proof}

\begin{thrm}\label{theorem.lambda}
Let $U_n\subset\Omega$ be a nested sequence so that $\mu(U_n)\to0$ as $n\to\infty$.
Assume that the limits $\hat\alpha_\ell(L)=\lim_{n\to\infty}\hat\alpha_\ell(L,U_n)$ exist for $\ell=1,2,\dots$
and $L$ large enough.
Assume $\sum_\ell\ell\hat\alpha_\ell<\infty$, then
$$
\lambda_k=\frac{\alpha_k-\alpha_{k+1}}{\alpha_1}
$$
where $\alpha_k=\hat\alpha_k-\hat\alpha_{k+1}$.
In particular the limit defining $\lambda_k$ exists.
\end{thrm}

\begin{proof}
Let $\varepsilon>0$ then there exists $k_0$ so that
$\sum_{\ell=k_0}^\infty\ell\hat\alpha_\ell<\varepsilon$. Moreover there exists $L_0$
so that $|\hat\alpha_\ell-\hat\alpha_\ell(L)|<\varepsilon/k_0$ for all $L\ge L_0$ and $\ell\in[1,k_0]$.
For $n$ large enough we also have $|\hat\alpha_\ell(L)-\hat\alpha_\ell(L,U_n)|<\varepsilon/k_0$.
In the following we will often write $U$ for $U_n$.

Let $L'>L$, then
$$
\mathbb{P}(Z^{L'}=k)
=\frac1k\mathbb{E}(\ind_{Z^{L'}=k}Z^{L'})
=\frac1k\sum_{i=0}^{2L'}\mathbb{E}(\ind_{Z^{L'}=k}\ind_{I_i=1}).
$$
For $i\in[L,2L'-L]$ put
$$
D_{i}^{L,L'}=\left\{\sum_{b=i+L+1}^{L'}I_b\ge1,\;I_i=1\right\}.
$$
By Lemma~\ref{tail.lemma} $\mu(D_i^{L,L'})=\mathcal{O}(\eta\mu(U))$ for $L$ big enough
and $n$ large enough,
where $\eta>0$ will be chosen below.
Let $k\ge1$, then
$$
\left\{W^{i+L}=k,\;I_i=1\right\}\cap\!\left(D_i^{L,L'}\right)^c
\subset \left\{Z^{L'}=k,\;I_i=1\right\}
$$
and also
$$
\{Z^{L'}=k,\;I_i=1\}\subset\{W^{i+L}=k,\;I_i=1\}\cup D_i^{L,L'}.
$$
These two inclusions imply
$$
\mathbb{P}(Z^{L'}=k,\;I_i=1)
=\mathbb{P}\!\left(W^{i+L}=k,\;I_i=1\right)+\mathcal{O}(\eta\mu(U)).
$$
Put
$$
R_{k,\ell}^{i,L}=\left\{\sum_{b=i}^{i+L}I_b=k-\ell,\; W^{i-1}=\ell,\;I_i=1\right\}
$$
for the set of $k$-clusters that have $\ell$ occurrences to the `left' of $i$.
Then
$$
R_{k,\ell}^{i,L}(j)
=R_{k,\ell}^{i,L}\cap\left\{I_{i-j}=1,\;I_a=0\forall a=0,\dots,i-j-1\right\}
$$
denotes all those $k$-clusters which have $\ell$ occurrences to the left of $i$ the
first one of which occurs $j$ steps to the left of $i$.
Evidently, $R_{k,\ell}^{i,L}=\bigcup_{j=1}^iR_{k,\ell}^{i,L}(j)$ is a disjoint union.
Let us note that the set
$$
F^{i-\frac{L}2}=\left\{W^{i-\frac{L}2}>0,\;I_i=1\right\}
$$
 has by Lemma~\ref{tail.lemma}
measure $\mathcal{O}(\eta\mu(U))$.
Then for every $\ell$ we obtain the inclusion
$$
R_{k,0}^{i,L}\cap \!\left(F^{i-\frac{L}2}\right)^c
\subset \bigcup_{j=i-\frac{L}2}^{i-1}T^{-j}R_{k,\ell}^{i,L}(j)
\subset R_{k,0}^{i,L}\cup D_i^{L,L'}\cup F^{i-\frac{L}2}
$$
where the union over $j$ is a disjoint union since
$T^{-j}R_{k,\ell}^{i,L}(j)\cap T^{-j'}R_{k,\ell}^{i,L}(j')=\varnothing$ if $j\not=j'$.
Thus for every $\ell=0,\dots,k-1$:
$$
\mu\!\left( \bigcup_{j=i-\frac{L}2}^{i-1}T^{-j}R_{k,\ell}^{i,L}(j)\right)
=\mu\!\left(R_{k,0}^{i,L}\right)+\mathcal{O}(\eta\mu(U))
$$
and since the union is disjoint this implies
$$
 \sum_{j=i-\frac{L}2}^{i-1}\mu(R_{k,\ell}^{i,L}(j))
 \le\mu(R_{k,\ell}^{i,L})
 \le \sum_{j=i-\frac{L}2}^{i-1}\mu(R_{k,\ell}^{i,L}(j))+\mu(F^{i-\frac{L}2})
$$
from which we conclude that
$$
\mu(R_{k,\ell}^{i,L})
 =\mu(R_{k,0}^{i,L})+\mathcal{O}(\eta\mu(U))
 =\mu(R_{k,0}^{L,L})+\mathcal{O}(\eta\mu(U))
$$
where the last step is due to invariance.
Therefore
\begin{eqnarray*}
\mathbb{P}(Z^{L'}=k)
&=&\frac1k\!\left(\sum_{i=L}^{2L'-L}\sum_{\ell=0}^{k-1}\!
\left(\mu(R_{k,\ell}^{i,L})+\mathcal{O}(\eta\mu(U))\right)+\mathcal{O}(2L\mu(U))\right)\\
&=&2L'\!\left(1-\frac{L}{L'}\right)\!\left(\mu(R_{k,0}^{L,L})+\mathcal{O}(\eta\mu(U))\right)+\mathcal{O}(L\mu(U)),
\end{eqnarray*}

In a similar way let us put
$$
S_{k,\ell}^{i,L}(j)
=R_{k,\ell}^{i,L}\cap\left\{I_{i-j}=1,\;I_a=0\forall a\in(i-j,i)\right\}
$$
for the set $k$-clusters which have $\ell$ occurrences to the left of $i$ the
last one of which occurs $j$ steps to the left of $i$.
As before we obtain
$$
R_{k,\ell-1}^{i,L}\cap \!\left(R^{i-\frac{L}2}\right)^c
\subset \bigcup_{j=i-\frac{L}2}^{i-1}T^{-j}S_{k,\ell}^{i,L}(j)
\subset R_{k,\ell-1}^{i,L}\cup D_i^{L,L'}\cup F^{i-\frac{L}2}
$$
and therefore conclude that
\begin{equation}\label{invariance}
\mu(R_{k,\ell}^{i.L})=\mu(R_{k,\ell-1}^{i.L})+\mathcal{O}(\eta\mu(U)).
\end{equation}

 Since
 $$
 \mathbb{P}(Z^{L,+}=k,I_L=1)
 =(1+\mathcal{O}(\varepsilon))\mu(U)\alpha_k
 $$
 we obtain
 \begin{eqnarray*}
 \alpha_k(L,U)-\alpha_{k+1}(L,U)
 &=&(1+\mathcal{O}(\varepsilon))\mu(U)^{-1}\!
 \left(\mathbb{P}(Z^{L,+}=k , I_L=1)-\mathbb{P}(Z^{L,+}=k+1 , I_L=1)\right)\\
  &=&(1+\mathcal{O}(\varepsilon))\mu(U)^{-1}
  \sum_{\ell=0}^\infty
  \!\left(\mu(R_{k+\ell,\ell}^{L,L})-\mu(R_{k+1+\ell,\ell}^{L,L})\right)\\
  &=&(1+\mathcal{O}(\varepsilon))\mu(U)^{-1}
  \sum_{\ell=0}^{k_0}
  \!\left(\mu(R_{k+\ell,\ell}^{L,L})-\mu(R_{k+1+\ell,\ell+1}^{L,L})+\mathcal{O}(\eta\mu(U))\right)\\
  &&+\mathcal{O}(\mu(U)^{-1})\sum_{\ell=k_0+1}^\infty \!\left(\mu(R_{k+\ell,\ell}^{L,L})+\mu(R_{k+1+\ell,\ell}^{L,L})\right)
  \end{eqnarray*}
 In order to estimate the tail sum $\sum_{\ell=k_0}^\infty\mu(R_{k+\ell,\ell}^{L,L})$
 we first notice that
 $$
 T^{-j}R_{k+\ell,\ell}^{L,L}(j)\cap  T^{-j'}R_{k+\ell',\ell'}^{L,L}(j')=\varnothing
 $$
 if $j=j', \ell\not=\ell'$ and also in the case when $j\not=j'$ and $|\ell'-\ell|>k$.
 To see the latter, assume $j'>j$ and
 $ T^{-j}R_{k+\ell,\ell}^{L,L}(j)\cap  T^{-j'}R_{k+\ell',\ell'}^{L,L}(j')\not=\varnothing$
 then the occurrences in $[i,i+j)$ are identical in both sets. Moreover,
 since the occurrences in $[i+j,i+j')$ are identical this forces not only
 $\ell'\ge\ell$ but also that $\ell'-\ell\le k$ since $T^{-j}R_{k+\ell,\ell}^{L,L}(j)$
 has exactly $k$ occurrences on $[i+j,i+k)$. (There are $k-(\ell'-\ell)$ occurrences
 on $[i+j',i+j+k]$ and for $T^{-j'}R_{k+\ell',\ell'}^{L,L}(j')$ there are $\ell'-\ell$
 occurrences on $(i+j+k,i+j'+k]$.)
If we choose an integer $k'>k$ then for every $p=0,1,\dots, k'-1$ one has
$$
\bigcup_{j=1}^iT^{-j}\bigcup_{s=\frac{k_0}{k'}}^\infty R_{k+sk'+p,sk'+p}^{L,L}(j)
\subset T^{-L}\left\{W^{2L}\ge k_0+k, I_0=1\right\}
$$
where the double union on the left hand side is disjoint. Therefore
 $$
\sum_{s=\frac{k_0}{k'}}^\infty \mu(R_{k+sk'+p,sk'+p}^{L,L})
\le \mathbb{P}(W^{2L}\ge k_0+k, I_0=1)
=\mu(U)\hat\alpha_{k_0+k}(2L,U).
$$
and consequently
 $$
\sum_{\ell=k_0}^\infty \mu(R_{k+\ell,\ell}^{L,L})
\le k'\mu(U)\hat\alpha_{k_0+k}(2L,U).
$$
The same estimate also applies to the tail sum of $\mu(R_{k+1+\ell,\ell}^{L,L})$.

This gives us
$$
\mu(R_{k,0}^{L,L})
=(1+\mathcal{O}(\varepsilon))\mu(U)(\alpha_k(L)-\alpha_{k+1}(L))
 +\mathcal{O}(k_0\eta\mu(U))+k'\mu(U)\hat\alpha_{k_0+k}(2L,U).
  $$
If we choose $\eta=\varepsilon/k_0$, $k'=k_0+k$ and $L'=L^\gamma$ for some $\gamma>1$, then
\begin{eqnarray*}
\mathbb{P}(Z^{L^\gamma}=k)
&=&2L^\gamma\mu(U)\!
\left(\!\left(1-L^{1-\gamma}\right)(1+\mathcal{O}(\varepsilon))(\alpha_k(L,U)-\alpha_{k+1}(L,U))\right.\\
&&\hspace{4cm}\left.+\mathcal{O}(\varepsilon)+\mathcal{O}(L^{1-\gamma})+(k_0+k)\hat\alpha_{k_0+k}(2L,U)\right),
\end{eqnarray*}

Without loss of generality we can assume that $L$ is large enough
 so that $L^{1-\gamma}<\varepsilon$. Then
 \begin{eqnarray*}
\mathbb{P}(Z^{L^\gamma}>0)
&=&\sum_{k=1}^\infty\mathbb{P}(Z^{L^\gamma}=k)\\
&=&2L^\gamma(	1+\mathcal{O}(\varepsilon))\mu(U)
\!\left(\sum_{k=1}^{k_0}(\alpha_k(L,U)-\alpha_{k+1}(L,U)
+\mathcal{O}(\varepsilon))+\sum_{\ell=k_0}^\infty\ell\hat\alpha_\ell(2L,U)\right)\\
&=&2L^\gamma(	1+\mathcal{O}(\varepsilon))\mu(U)
\!\left(\alpha_1(L,U)+\mathcal{O}(\varepsilon)\right)
   \end{eqnarray*}
where the tail sum on the RHS is estimated by $2\varepsilon$. Hence
 $$
 \mathbb{P}(Z^{L^\gamma}>0)
=2L^\gamma(	1+\mathcal{O}(\varepsilon))\mu(U)
(\alpha_{1}(L,U)+\mathcal{O}(\varepsilon)).
 $$

Combining the two estimates yields
$$
\lambda_k(L^\gamma,U_n)
=\frac{\mathbb{P}(Z^{L^\gamma}=k)}{\mathbb{P}(Z^{L^\gamma}>0)}
=(	1+\mathcal{O}(\varepsilon))
\frac{\alpha_k(L,U_n)-\alpha_{k+1}(L,U_n)+\mathcal{O}(\varepsilon)}{\alpha_1(L,U_n)+\mathcal{O}(\varepsilon)}.
$$
Letting  $\varepsilon\to0$ implies $L\to\infty$ and consequently $\mu(U_n)\to0$ as $n\to\infty$
let us finally obtain (as $\gamma>1$) as claimed
$\lambda_k=(\alpha_k-\alpha_{k+1})/\alpha_1$.
\end{proof}

\begin{remark} Under the assumption of Theorem~\ref{theorem.lambda} the
 expected length of the clusters is
$$
\sum_{k=1}^\infty k\lambda_k=\frac{1}{{\alpha_1}}\sum_{k=1}^{\infty}k(\alpha_k-\alpha_{k+1})=\frac1{\alpha_1}
$$
which is the reciprocal of the extremal index $\alpha_1$.
\end{remark}

\begin{remark}
Since $\lambda_k\ge0$ we conclude that $\alpha_1\ge\alpha_2\ge\alpha_3\ge\cdots$ is
a decreasing sequence. It is moreover easy to see that
$\lambda_k=\alpha_k\forall k$ only when both are geometrically distributed,
 i.e.\ when $\lambda_k=\alpha_k=\alpha_1(1-\alpha_1)^{k-1}$.
 Also notice that the condition $\sum_kk\hat\alpha_k<\infty$ of the theorem is equivalent to
  $\sum_kk^3\lambda_k<\infty$ or $\sum_kk^2\alpha_k<\infty$.
\end{remark}

\begin{cor} For every $\eta>$ one has
$$
\left|\mathbb{P}(Z_i^{L,-}=k, Z_i^{L,+}=\ell-k,I_i=1)
-\mathbb{P}(Z_i^{L,-}=k', Z_i^{L,+}=\ell-k',I_i=1)\right|
\le\eta\mu(U_n)
$$
for all $0\le k,k'<\ell$, provided $L$ and $n$ are large enough.
\end{cor}

\begin{proof}
This follows from~\eqref{invariance} as
$\mathbb{P}(Z_i^{L,-}=k, Z_i^{L,+}=\ell-k,I_i=1)=\mu(R^{i,L}_{k,\ell})$.
\end{proof}

\section{Entry times} \label{entry.times}

Let us consider the entry time $\tau_U(x)$ where $x\in\Omega$.

\begin{lem}\label{lemma.entry}
Let $U_n\subset\Omega$ be a nested sequence so that $\mu(U_n)\to0$ as $n\to\infty$.
Assume that the limits $\hat\alpha_\ell(L)=\lim_{n\to\infty}\hat\alpha_\ell(L,U_n)$ exist for $\ell=1,2,\dots$
and $L$ large enough.

Then
$$
\lim_{L\to\infty}\lim_{n\to\infty}\frac{\mathbb{P}(\tau_{U_n}\le L)}{L\mu(U_n)}=\alpha_1.
$$
\end{lem}

\begin{proof}
If we write again $U$ for $U_n$ then
$$
\mathbb{P}(\tau_U\le L)
=\mu\!\left(\bigcup_{j=0}^{L}T^{-j}U\right)
=L\mu(U)-\sum_{\ell=2}^{L}\mathbb{P}(\tau_U^\ell\le L<\tau_U^{\ell+1})
$$
where
$$
\mathbb{P}(\tau_U^\ell\le L<\tau_U^{\ell+1})
=\sum_{j=0}^{L-\ell}\mathbb{P}(\tau_U^\ell\le L<\tau_U^{\ell+1}, \tau=j)
$$
and for $j=0,1,\dots,L-\ell$:
$$
\mathbb{P}(\tau_U^\ell\le L<\tau_U^{\ell+1}, \tau=j)
=\mathbb{P}(\tau_U^{\ell-1}\le L-j<\tau_U^\ell|U)\mu(U)
$$
for which we use
$$
\mathbb{P}(\tau_U^{\ell-1}\le L-j<\tau_U^\ell|U)=
(1+o(1))\alpha_\ell(L-j)
$$
as $n\to\infty$.
Let $\varepsilon>0$ then there exists $L_0$ so that $\hat\alpha_\ell<\varepsilon$ for all $\ell\ge L_0$.
This implies that $\sum_{\ell=L_0}^\infty\alpha_\ell(L)<\varepsilon$ for all $L$
as $\hat\alpha_\ell\ge\hat\alpha_\ell(L)\forall L$ and the $\hat\alpha_\ell$ are the tail sums of
the $\alpha_\ell$.
Moreover
$\alpha_\ell(L-j)=(1+\mathcal{O}^*(\varepsilon))\alpha_\ell(L)$
for all $j\le L-\sqrt{L}$ and $\ell\le L_0$ if $L$ is large enough.
Therefore
$$
\mathbb{P}(\tau_U^{\ell-1}\le L-j<\tau_U^\ell)
=(1+o(1))\alpha_\ell(L)
=(1+o(1))\alpha_\ell
$$
and
$$
\mathbb{P}(\tau_U^\ell\le L<\tau_U^{\ell+1})
=(1+o(1))L(1+\mathcal{O}(L^{-\frac12}))\mu(U)\alpha_\ell.
$$

Thus
$$
\mathbb{P}(\tau_U\le L)
=L\mu(U)\!\left(1-(1+o(1))(1+\mathcal{O}(L^{-\frac12}))\!\left(\sum_{\ell=2}^{L_0}\alpha_\ell+\varepsilon\right)\right)
$$
and therefore
$$
\lim_{n\to\infty}\frac{\mathbb{P}(\tau_{U_n}\le L)}{L\mu(U_n)}
=\alpha_1+\mathcal{O}(L^{-\frac12})+\mathcal{O}(\varepsilon)
=\alpha_1+\mathcal{O}(\varepsilon)
$$
for $L$ large enough. The statement of the lemma now follows by as $\varepsilon\to0$
which implies $L\to\infty$.
\end{proof}

\begin{remark}
In a similar way as in the previous lemma on can show for $\ell=2,3,\dots$ that
\begin{eqnarray*}
\mathbb{P}(\tau_{U_n}^\ell\le L)
&=&\sum_{j=0}^{L-\ell}\mathbb{P}(\tau_{U_n}^\ell\le L,\tau_{U_n}=j)\\
&=&\sum_{j=0}^{L-\ell}\mathbb{P}(\tau_{U_n}^{\ell-1}\le L-j|U_n)\mu(U_n)\\
&=&\mu(U_n)\sum_{j=0}^{L-\ell}\hat\alpha_\ell(L-j)(1+o(1))
\end{eqnarray*}
which implies as before that
$$
\lim_{L\to\infty}\lim_{n\to\infty}\frac{\mathbb{P}(\tau_{U_n}^\ell\le L)}{L\mu(U_n)}=\hat\alpha_\ell.
$$
Also
$$
\lim_{L\to\infty}\lim_{n\to\infty}\frac{\mathbb{P}(\tau_{U_n}^\ell\le L<\tau_{U_n}^{\ell+1})}{L\mu(U_n)}=\alpha_\ell
$$
for $\ell=2,3,\dots$.
\end{remark}

\section{The Compound Binomial Distribution} \label{binomial}

This section contains the abstract  approximation theorem which establishes the distance
between sums of $\{0,1\}$-valued dependent random variables $X_n$ and a random variable that
has a compound Binomial distribution.
It is used in Section~\ref{set_up_T1} in the proof of Theorem~1 and compares
 the number of occurrences in a finite time interval with the number of occurrences
 in the same interval
 for a compound binomial process.

 Let $Y_j$ be $\mathbb{N}$ valued i.i.d.\ random variables and
 denote $\lambda_\ell=\mathbb{P}(Y_j=\ell)$.
 Let $N$ be a (large) positive integer, $s>0$ a parameter and put $p=s/N$.
 If $Q$ is a binomially distributed random variable with parameters $(N,p)$,
 that is $\mathbb{P}(Q=k)=\binom{N}kp^k(1-p)^{N-k}$,
 then $W=\sum_{i=1}^QY_i$ is compound binomially distributed.
 The generating function of $W$ is
 $\varphi_W(z)
 =\left(p(\varphi_{Y_1}(z)-1)+1\right)^N$, where
 $\varphi_{Y_1}(z)=\sum_{\ell=0}^\infty z^\ell\lambda_\ell$
 is the generating function of $\tilde{Z}_1$. As $N$ goes to infinity, $Q$ converges to
 a Poisson distribution with parameter $s$ and $W$ converges to a compound
 Poisson distribution with parameters $s\lambda_\ell$. In particular
 $\varphi_W(z)\rightarrow \exp s(\varphi_{Y_1}(z)-1)$.
 (In the following theorem we assume for simplicity's sake that $N'$ and $\Delta$ are integers.)

\begin{thrm} \label{helper.theorem}
Let $(X_n)_{n \in \mathbb{N}}$ be a stationary $\{0,1\}$-valued process and
  $W = \sum_{i=0}^N X_i$ for some (large) integer $N$.
 Let $K, \Delta$ be positive integers so that $\Delta(2K+1)<N$ and
 define $Z=\sum_{i=0}^{2K}X_i$ and $W_a^b=\sum_{i=a}^b X_i$
 ($W=W_0^N$). Let $\tilde\nu$
 be the compound binomial distribution measure where the binomial part
 has values $p=\mathbb{P}(Z\ge1)$ and $N'=N/(2K+1)$ and the compound part has
 probabilities  $\lambda_\ell=\mathbb{P}(Z=\ell)/p$ .
Then there exists a constant $C_3$, independent of $K$ and $\Delta$,
such that
$$
\abs{\mathbb{P}(W=k) - \tilde\nu(\{k\})}
 \leq C_3(N'(\mathcal{R}_1 + \mathcal{R}_2) + \Delta\mathbb{P}(X_0=1)),
$$
where
\begin{eqnarray*}
\mathcal{R}_1 &=& \sup_{\substack {0< \Delta<M\le N'\\0 <q<N'-\Delta-1/2}}
 \left|\sum_{u=1}^{q-1}\!\left(\mathbb{P}\!\left(Z=u \land W_{\Delta(2K+1)}^{M(2K+1)}=q-u\right)
-\mathbb{P}(Z=u)\mathbb{P}\!\left(W_{\Delta(2K+1)}^{M(2K+1)}=q-u\right)\right)\right|\\
\mathcal{R}_2 &=& \sum_{n=2}^\Delta \mathbb{P}(Z\ge1 \land Z\circ T^{(2K+1)n}\ge1).
\end{eqnarray*}
\end{thrm}

\begin{proof}
Let us assume for simplicity's sake that $N$ is a multiple of $2K+1$ and put $N'=N/(2K+1)$.
Now put $Z_j=\sum_{i=j(2K+1)}^{(2K+1)-1}X_i=Z\circ T^{j(2K+1)}$ for $j=0,1,\dots,N'$.
Thus   $V=\sum_{i=0}^NX_i=\sum_{j=0}^{N'}Z_j$.
 Let $(\tilde{Z}_j)_{j \in \mathbb{N}}$ be a sequence of independent, identically distributed random variables taking values in $\mathbb{N}_0$ which have the same distribution as $Z_j$.
Moreover let us put
$V_k^\ell=\sum_{j=k}^\ell Z_j$ and similarly $\tilde{V}_k^\ell=\sum_{j=k}^\ell \tilde{Z}_j$.
We have to estimate the following quantity:
$$
\mathbb{P}(V_0^{N'}=k)-\mathbb{P}(\tilde{V}_0^{N'}=k)
=\sum_{j=0}^{N'-1}D_j(k),
$$
where
\begin{eqnarray*}
D_j(k)
&=&\mathbb{P}(\tilde{V}_0^{j-1}+V_j^{N'}=k)-\mathbb{P}(\tilde{V}_0^{j}+V_{j+1}^{N'}=k)\\
&=&\sum_{\ell=0}^k\mathbb{P}(\tilde{V}_0^{j-1}=\ell)\!\left(\mathbb{P}(V_j^{N'}=k-\ell)
-\mathbb{P}(\tilde{Z}_j+V_{j+1}^{N'}=k-\ell)\right)
\end{eqnarray*}
By invariance it suffices to estimate
$$
\mathbb{P}(V_0^M=q)-\mathbb{P}(\tilde{Z}_0+V_1^M=q)
=\sum_{u=0}^q\mathcal{R}(u)
$$
for every $M\le N'$ and $q$,
where
$$
\mathcal{R}(u)
=\mathbb{P}(Z_0=u, V_1^M=q-u)
-\mathbb{P}(\tilde{Z}_0=u)\mathbb{P}(V_1^M=q-u).
$$
Let us first single out the terms $u=q$ and $u=0$. For $u=0$ we see that
$$
\mathbb{P}(Z_0=0, V_1^M=q)
=\mathbb{P}(V_1^M=q)-\mathbb{P}(Z_0\ge1, V_1^M=q)
$$
and
$$
\mathbb{P}(Z_0=0)\mathbb{P}(V_1^M=q)
=\mathbb{P}(V_1^M=q)-\mathbb{P}(Z_0\ge1)\mathbb{P}(V_1^M=q).
$$
Consequently
\begin{eqnarray*}
\mathcal{R}(0)
&=&\mathbb{P}(Z_0=0, V_1^M=q)-\mathbb{P}(\tilde{Z}_0=0)\mathbb{P}(V_1^M=q)\\
&=&\mathbb{P}(Z_0\ge1)\mathbb{P}(V_1^M=q)-\mathbb{P}(Z_0\ge1, V_1^M=q)\\
&\le&\sum_{u=1}^q\mathcal{R}(u).
\end{eqnarray*}
Similarly one obtains for $u=q$:
$$
\mathcal{R}(q)
=\mathbb{P}(Z_0=q)\mathbb{P}(V_1^M\ge1)-\mathbb{P}(Z_0=q, V_1^M\ge1).
$$
This implies that
$$
|\mathcal{R}|\le4\sum_{u=1}^{q-1}|\mathcal{R}(u)|.
$$
In order to estimate $|\mathcal{R}(u)|$ for $u=1,2,\dots,q-1$ let
$0\le\Delta<M$ be the length of the gap we will now introduce. Then
$$
\left|\mathcal{R}(u)\right|
\le \mathcal{R}_1(u)+\mathcal{R}_2(u)+\mathcal{R}_3(u),
$$
where
\begin{eqnarray*}
\mathcal{R}_1&=&\max_{\substack{\Delta<M\le N'\\q}}\left|\sum_{u=1}^{q-1}\!\left(\mathbb{P}(Z_0=u, V_\Delta^M=q-u)-\mathbb{P}(\tilde{Z}_0=u)\mathbb{P}(V_\Delta^M=q-u)\right)\right|.
\end{eqnarray*}
The other two terms $\mathcal{R}_2$ and $\mathcal{R}_3$ account for opening a `gap'.
More precisely
$$
\mathcal{R}_2(u)=\left|\mathbb{P}(Z_0=u, V_\Delta^M=q-u)-\mathbb{P}(Z_0=u, V_1^M=q-u)\right|
$$
and since
$$
\{Z_0=u, W_1^M=q-u\}\setminus\{Z_0=u, W_\Delta^M=q-u\}
\subset \{Z_0=u, W_1^{\Delta-1}\ge1\}
$$
we get therefore
$$
\sum_{u=1}^{q-1}\mathcal{R}_2(u)\le\mathbb{P}(Z_0\ge1, W_1^{\Delta-1}\ge1).
$$
For the third term we get
\begin{eqnarray*}
\mathcal{R}_3(u)&=&\mathbb{P}(\tilde{Z}_0=u)
\left|\mathbb{P}(V_1^M=q-u)-\mathbb{P}(V_\Delta^M=q-u)\right|.
\end{eqnarray*}
To estimate $\mathcal{R}_3$ observe that ($q'=q-u$)
$$
\mathbb{P}(V_1^M=q')
=\mathbb{P}(Z_1\ge1, V_1^M=q')+\mathbb{P}(Z_1=0, V_1^M=q')
$$
where
$$
\mathbb{P}(Z_1=0, V_1^M=q')
=\mathbb{P}(Z_1=0, V_2^M=q')
=\mathbb{P}(V_2^M=q')-\mathbb{P}(Z_1\ge1, V_2^M=q').
$$
Hence
$$
\mathbb{P}(V_1^M=q')-\mathbb{P}(V_2^M=q')
=\mathbb{P}(Z_1\ge1, V_1^M=q')-\mathbb{P}(Z_1\ge0, V_2^M=q')
$$
which implies more generally
$$
|\mathbb{P}(V_j^M=q')-\mathbb{P}(V_{j+1}^M=q')|
\le \mathbb{P}(Z_j\ge1)\le (2K+1)\mathbb{P}(X_0=1)
$$
for any $k=1,\dots,\Delta$.
Hence
$$
|\mathbb{P}(V_1^M=q')-\mathbb{P}(V_\Delta^M=q')|
\le\sum_{j=1}^{\Delta-1}\mathbb{P}(Z_j\ge1)
\le (2K+1)\Delta\mathbb{P}(X_0=1)
$$
and thus
$$
\sum_{u=1}^{q-1}\mathcal{R}_3(u)
\le (2K+1)\Delta\mathbb{P}(X_0=1)\sum_{u=1}^{q-1}\mathbb{P}(Z_0=u)
\le (2K+1)\Delta\mathbb{P}(X_0=1)^2
$$
since
$\{Z_0\ge1\}\subset\bigcup_{j=0}^{2K}\{X_j=1\}$.
We now can estimate one of the gap terms:
$$
\mathcal{R}_3
=\sum_{u=1}^{q-1}\mathcal{R}_3(u)
\le 2(2K+1)\Delta\mathbb{P}(X_0=1)^2
$$
for all $q$ and $M$.

Finally, from the previous estimates we obtain for  $k \leq N$,
$$
\abs{\mathbb{P}(V_0^{N'}=k)-\mathbb{P}(\tilde{V}_0^{N'}=k)}
\leq \const N'(\mathcal{R}_1+\mathcal{R}_2+K\Delta \mathbb{P}(X_0=1)^2).
$$
Since $N\mathbb{P}(X_0=1)=\mathcal{O}(1)$ and $N'=N/(2K+1)$ we obtain the RHS in
 the theorem.

 It remains to show that $\mathbb{P}(\tilde{V}_0^{N'}=k)=\tilde\nu(\{k\})$. To see
 this put $p=\mathbb{P}(\tilde{Z}_1\ge 1)$  and let $Y_j$ be $\mathbb{N}$-valued
 i.i.d.\ random variables with distribution $\mathbb{P}(Y_j=\ell)=\frac1p\mathbb{P}(\tilde{Z}_j=\ell)=\lambda_\ell$
 for $\ell=1,2,\dots$. Then $Q=|\{i\in[0,N']:\; \tilde{Z}_i\not=0\}|$ is binomially distributed
 with parameters $(N',p)$ and consequently $\tilde{V}=\sum_{i=1}^QY_i$ is
 compound binomial.
 \end{proof}

\section{Proof of Theorem 1}\label{proof.theorem1}

\subsection{Compound binomial approximation of the return times distribution} \label{set_up_T1}
To prove Theorem~1 we will employ the  approximation theorem from Section~\ref{binomial}
where we put $U=B_\rho(\Gamma)$.
 Let $X_i=\ind_{U} \circ T^{i-1}$, then we put $N = \floor{t/\mu(U)}$,
 where $t$ is a positive parameter. Let $K$ be an integer and put as before
 $V_a^b=\sum_{j=a}^bZ_j$, where the $Z_j=\sum_{i=j(2K+1)}^{(j+1)(2K+1)-1}X_i$
 are stationary random variables.
Then for any $2 \leq \Delta \leq N'=N/(2K+1)$ (for simplicity's sake we assume
$N$ is a multiple of $2K+1$)
\begin{equation} \label{errorUSE}
\Abs{\mathbb{P}(V_0^{N'}=k) - \tilde\nu(\{k\})} \;
\leq \; C_3 ( N'(\mathcal{R}_1+\mathcal{R}_2) + \Delta \mu(U)),
\end{equation}
where
\begin{align*}
\mathcal{R}_1 &= \sup_{\substack {0< \Delta<M\le N'\\0 <q<N'-\Delta-1/2}}
 \left|\sum_{u=1}^{q-1}\!\left(\mathbb{P}\!\left(Z_0=u \land V_{\Delta}^{M}=q-u\right)
-\mathbb{P}(Z_0=u)\mathbb{P}\!\left(V_{\Delta}^{M}=q-u\right)\right)\right|\\
\mathcal{R}_2 &= \sum_{j=1}^{\Delta} \mathbb{P}(Z_1\ge1\land Z_j\ge1),
\end{align*}
and $\tilde\nu$ is the compound binomial distribution with parameters
$p=\mathbb{P}(Z_j\ge1)$ and distribution $\frac{t}p\mathbb{P}(Z_j=k)$.
Notice that $\mathbb{P}(V_0^{N'}=k)=0$ for $k>N$ and also
$\tilde\nu(\{k\})=\mathbb{P}(\tilde{V}_0^{N'}=k)=0$ for $k>N$.

\vspace{0.5cm}
\noindent We now proceed to estimate the error between the distribution of $S$ and a
compound binomial  based on Theorem~\ref{helper.theorem}.

\subsection{Estimating $\mathcal{R}_1$} \label{est_R1_section}

Let us fix $\rho$ for the moment and put $U=B_\rho(\Gamma)$.
Fix $q$ and $u$ and we want to estimate the quantity
$$
\mathcal{R}_1(q,u) =
 \left|\mathbb{P}\!\left(Z_0=u, V_{\Delta}^{M}=q-u\right)
-\mathbb{P}(Z_0=u)\mathbb{P}\!\left(V_{\Delta}^{M}=q-u\right)\right|
$$
In order to use the decay of correlations~(II) to obtain an estimate for
$\mathcal{R}_1(q,u)$ we approximate $\ind_{Z_0=u}$ by Lipschitz functions
from above and below as follows. Let $r>0$ be small ($r<\!\!<\rho$) and put
$U''(r)=B_{r}(U)$ for the outer approximation of $U$ and
$U'(r)=(B_{r}(U^c))^c$ for the inner approximation.
We then consider the set $\mathcal{U}=\{Z_0=u\}$ which is a disjoint
union of sets
$$
\bigcap_{j=1}^uT^{-v_j}U\cap\bigcap_{i\in[0,2K+1]\setminus\{v_j:j\}}T^{-i}U^c
$$
where $0\le v_1<v_2<\cdots<v_u\le 2K+1$ the $u$ entry times vary over all
possibilities. Similarly we get its outer approximation $\mathcal{U}''(r)$ and
its inner approximation $\mathcal{U}'(r)$ by using $U''(r)$ and $U'(r)$ respectively.
 We now consider Lipschitz continuous
functions approximating $\ind_{\mathcal{U}}$ as follows
\begin{equation*}
\phi_r(x) =
\begin{cases}
1 & \text{on $\mathcal{U}$} \\
0 & \text{outside $\mathcal{U}''(r)$}
\end{cases}
\hspace{0.7cm} \text{and} \hspace{0.7cm}
\tilde{\phi}_r(x) =
\begin{cases}
1 & \text{on $\mathcal{U}'(r)$} \\
0 & \text{outside $\mathcal{U}$}
\end{cases}
\end{equation*}
with both linear in between. The Lipschitz norms of both $\phi_r$ and $\tilde{\phi}_r$
 are bounded by $a^{2K+1}/r$ where $a=\sup_{x\in\mathcal{G}}|DT(x)|$.
  By design $\tilde{\phi}_r \leq \ind_{Z_0=u} \leq \phi_r$.

We obtain
\begin{align*}
\mathbb{P}\!\left(Z_0=u, V_{\Delta}^{M}=q-u\right)
-\mathbb{P}(Z_0=u)\mathbb{P}\!\left(V_{\Delta}^{M}=q-u\right)\hspace{-3cm} \\
& \leq \int_M \phi_r \; \cdot \ind_{V_{\Delta}^{M}=q-u} \, d\mu - \int_M \ind_{Z_0=u} \, d\mu \, \int_M \ind_{V_{\Delta}^{M}=q-u} \, d\mu \\[0.2cm]
& =X+Y
\end{align*}
where
\begin{align*}
X&=\left(\int_M \phi_r \, d\mu - \int_M \ind_{Z_0=u} \, d\mu \right)
 \int_M \ind_{V_{\Delta}^{M}=q-u} \, d\mu\\
Y&=\int_M \phi_r \; (\ind_{V_{\Delta}^{M}=q-u} ) \, d\mu
- \int_M \phi_r \, d\mu \, \int_M \ind_{V_{\Delta}^{M}=q-u} \, d\mu .
\end{align*}
The two terms $X$ and $Y$ are estimated separately.
The first term is readily estimated by:
$$
X \leq\mathbb{P}(V_{\Delta}^{M}=q-u)  \, \int_M (\phi_r - \ind_{Z_0=u}) \, d\mu
 \leq \mu(\mathcal{U}''(r) \setminus \mathcal{U}(r)).
$$
In order to estimate the second term $Y$ we use the decay of correlations.
For this we have to approximate $\ind_{V_{0}^{M-\Delta}=q-u}$ by a function which is
constant on local stable leaves. (Note that if the map is expanding then there are no stable
leaves and $Y$ is straighforwardly estimated by $\mathcal{C}(\Delta)\|\phi_r\|_{Lip}$
as $\|\ind_{V_\delta^M=q-u}\|_\infty=1$.)
Let us define
$$
\mathcal{S}_n
=\bigcup_{\substack{\gamma^s\\T^n\gamma^s\subset U}}T^n\gamma^s,
\hspace{6mm}
\partial\mathcal{S}_n
=\bigcup_{\substack{\gamma^s, T^n\gamma^s\not\subset U\\T^n\gamma^s\cap U\not=\varnothing}}T^n\gamma^s
$$
and
$$
\mathscr{S}_\Delta^{M}=\bigcup_{n=\Delta(2K+1)}^{M(2K+1)}\mathcal{S}_n,
\hspace{6mm}
\partial\mathscr{S}_\Delta^{M}=\bigcup_{n=\Delta(2K+1)}^{M(2K+1)}\partial\mathcal{S}_n.
$$
The set
$$
\mathscr{S}_\Delta^{M}(q)=\{V_0^{M-\Delta}=q-u\}\cap\mathscr{S}_\Delta^{M}
$$
is then a union of local stable leaves. This follows from the fact that by construction
$T^n y\in U$ if and only if $T^n\gamma^s(y)\subset U$.
We also have
$\{V_0^{M-\Delta}=q-u\}\subset\tilde{\mathscr{S}}_\Delta^{M}(q)$
where the set $\tilde{\mathscr{S}}_\Delta^{M}(k)=\mathscr{S}_\Delta^{M}(k)\cup\partial\mathscr{S}_\Delta^{M}$
is a union of local stable leaves.

Denote by  $\psi_\Delta^{M}$ the characteristic function of the set $\mathscr{S}_\Delta^{M}(k)$
and by $\tilde\psi_\Delta^{M}$ the characteristic function for
$\tilde{\mathscr{S}}_\Delta^{M}(k)$. Then $\psi_\Delta^{M}$ and $\tilde\psi_\Delta^{M}$
are constant on local stable leaves and satisfy
$$
\psi_\Delta^{M}\le\ind_{V_0^{M-\Delta}=q-u}\le\tilde\psi_\Delta^{M}.
$$
Since $\{y:\psi_\Delta^{M}(y)\not=\tilde\psi_\Delta^{M}(y)\}\subset\partial\mathscr{S}_\Delta^{M}$
we need to estimate the measure of $\partial\mathscr{S}_\Delta^{M}$.

 By the contraction property
  $\mbox{diam}(T^n\gamma^s(y))\le \delta(n)\lesssim n^{-\kappa}$ outside the set $\mathcal{G}_n^c$  and consequently
  $$
  \bigcup_{\substack{\gamma^s\subset \mathcal{G}_n\\T^n\gamma^s\not\subset U\\
 T^n\gamma^s\cap U\not=\varnothing }}
  T^n\gamma^s \subset U''(\delta(n))\setminus U'(\delta(n)).
  $$
Therefore
\begin{eqnarray*}
\mu(\partial\mathscr{S}_\Delta^{M})
&\le&\mu\left(\bigcup_{n=\Delta(2K+1)}^{M(2K+1)}
T^{-n}\left(U''(\delta(n))\setminus U'(\delta(n))\right)\right)
+\sum_{n=\Delta(2K+1)}^\infty\mu(\mathcal{G}_{n}^c)\\
&\le&\sum_{n=\Delta(2K+1)}^{M(2K+1)}\mu(U''(\delta(n))\setminus U'(\delta(n)))
+\sum_{n=\Delta(2K+1)}^\infty\mu(\mathcal{G}_{n}^c)
\end{eqnarray*}
where the last term is estimated by assumption~(III) as follows
$$
\sum_{n=\Delta(2K+1)}^\infty\mu(\mathcal{G}_{n}^c)
=\mathcal{O}(1)\sum_{n=\Delta(2k+1)}^\infty n^{-q}
=\mathcal{O}(K^{-q}\Delta^{-q+1})
=\mathcal{O}(K^{-q}\rho^\epsilon\mu(U))
$$
where we put $\Delta\sim\rho^{-\upsilon}$ and we also assume that $q$ satisfies
$\upsilon(q-1)>d_1$ (that is $\epsilon=\upsilon(q-1)-d_1>0$).
Now by assumption~(VI):
$$
\mu(\partial\mathscr{S}_\Delta^{M})
=\mathcal{O}(1) \sum_{n=\Delta}^{\infty}\frac{n^{-\kappa\eta}}{\rho^\beta}\mu(U)
+ \rho^\epsilon\mu(U)
=\mathcal{O}(\rho^{v(\kappa\eta-1) -\beta}+\rho^\epsilon) \mu(U)
$$
with $\delta(n) = \mathcal{O}(n^{-\kappa})$ and $\Delta\sim\rho^{-v}$ where this time we
also need that $v > \frac{\beta}{\kappa\eta-1}$ which is determined in Section~\ref{total_error}  below.
Both constraints imply that we must have $v>\max\{\frac{d_1}{q-1},\frac\beta{\kappa\eta-1}\}$.
If we split $\Delta=\Delta'+\Delta''$
then, using assumption~(II), we can estimate as follows:
\begin{align*}
Y&=\left|  \int_M \phi_r \; T^{-\Delta'}(\ind_{V_{\Delta''}^{M-\Delta'}=q-u}  ) \, d\mu
- \int_M \phi_r   \, d\mu \, \int_M \ind_{V_0^{M-\Delta}=q-u} \, d\mu\right|\\
&
\le \mathcal{C}(\Delta')\|\phi_r\|_{Lip}\|\ind_{\tilde{\mathscr{S}}_{\Delta''}^{M-\Delta'}}\|_{\mathscr{L}^\infty}
+2\mu(\partial\mathscr{S}_{\Delta''}^{M-\Delta'}).
  \end{align*}
Hence
\begin{eqnarray*}
\mu(\mathcal{U} \cap T^{-\Delta}\{V_{0}^{M-\Delta}=q-u\})
- \mu(\mathcal{U}) \, \mathbb{P}(V_{0}^{M-\Delta}=q-u)
\hspace{-4cm}&&\\
&\leq &a^{2K+1}\frac{\mathcal{C}(\Delta/2)}{ r} + \mu(\mathcal{U}(r)\setminus \mathcal{U} )
+\mu(\partial\mathscr{S}_\Delta^{M})
\end{eqnarray*}
by taking $\Delta' = \Delta'' = \frac{\Delta}{2}$.
A similar estimate from below can be done using $\tilde\phi_\rho$. Hence
\begin{eqnarray*}
\mathcal{R}_1
&\leq &c_2\left(a^{2K+1}\frac{\mathcal{C}(\Delta/2)}{r}
+\mu(\mathcal{U}''(r)\setminus\mathcal{U}'(r))\right)
+\mu(\partial\mathscr{S}_\Delta^{M})\\
&\lesssim &a^{2K+1}\frac{\mathcal{C}(\Delta/2)}{r}
+\mu(\mathcal{U}''(r)\setminus\mathcal{U}'(r))
+(\rho^{v(\kappa\eta-1) -\beta}+\rho^\epsilon) \mu(U).
\end{eqnarray*}

In the exponential case when $\delta(n)=\mathcal{O}(\vartheta^n)$  we choose
$\Delta=s\abs{\log\rho}$ for some $s>0$ and obtain the estimate
$$
\mathcal{R}_1
\leq c_2\left(a^{2K+1}\frac{\mathcal{C}(\Delta/2)}{r}
+\mu(\mathcal{U}''(r)\setminus\mathcal{U}'(r)))\right)
+\mathcal{O}(\rho^{s\abs{\log\vartheta}-\beta}+\rho^\epsilon)\mu(U).
$$

\subsection{Estimating the  terms  $\mathcal{R}_2$}

We will first estimate the measure of $U\cap T^{-j}U$ for positive $j$.
Fix $j$ and and let $\gamma^u$ be an unstable local leaf through $U$.
Let us define
$$
\mathscr{C}_j(U,\gamma^u)=\{\zeta_{\varphi,j}: \zeta_{\varphi,j}\cap U\not=\varnothing,\varphi\in \mathscr{I}_j\}
$$
 for the cluster of $j$-cylinders that covers the set $U$.
 As before the sets $\zeta_{\varphi,k}$ are $\varphi$-pre-images of embedded $R$-balls in
 $T^j\gamma^u$.
Then
\begin{eqnarray*}
\mu_{\gamma^u}(T^{-j}U\cap U)
&\le&\sum_{\zeta\in\mathscr{C}_j(U,\gamma^u)}\frac{\mu_{\gamma^u}(T^{-j}U\cap \zeta)}{\mu_{\gamma^u}(\zeta)}\mu_{\gamma^u}(\zeta)\\
&\le&\sum_{\zeta\in \mathscr{C}_j(U,\gamma^u)}c_3\omega(j)
\frac{\mu_{T^j\gamma^u}(U\cap T^j\zeta)}
{\mu_{T^j\gamma^u}(T^j\zeta)}\mu_{\gamma^u}(\zeta)
\end{eqnarray*}
The denominator is uniformly bounded from below because
$\mu_{T^j\gamma^u}(T^j\zeta)=\mu_{T^j\gamma^u}(B_{R,\gamma^u}(y_k))$
for some $y_k$. Thus, by assumption (I), we have:
\begin{eqnarray*}
\mu_{\gamma^u}(T^{-j}U\cap U)
&\le& c_4\omega(j)\mu_{T^j\gamma^u}(U)
\sum_{\zeta\in \mathscr{C}_j(U,\gamma^u)}\mu_{\gamma^u}(\zeta)\\
&\le& c_4\omega(j)\mu_{T^j\gamma^u}(U)\,  L\,
\mu_{\gamma^u}\!\left(\bigcup_{\zeta\in \mathscr{C}_j(U,\gamma^u)}\zeta\right)
\end{eqnarray*}
Now, since outside the set $\mathcal{G}_n^c$ one has
 $$
 \bigcup_{\zeta\in \mathscr{C}_j(U,\gamma^u)}\zeta
\subset B_{j^{-\kappa}}(U)
$$
where by assumption $\mu_{\gamma^u}(B_{j^{-\kappa}}(U))=\mathcal{O}( j^{-\kappa u_0})$.
Therefore
$$
\mu_{\gamma^u}\left(\bigcup_{\zeta\in\mathcal{C}_j(U,\gamma^u)}\zeta\right)
\le \mathcal{O}(j^{-\kappa u_0})+\mu(\mathcal{G}_j^c)
=\mathcal{O}(j^{-\kappa u_0}+j^{-q})
$$
and consequently
$$
\mu_{\gamma^u}(T^{-j}U\cap U)\le c_5\omega(j)\mu_{T^j\gamma^u}(U)(j^{-\kappa u_0}+j^{-q}).
$$
Since   $d\mu=d\mu_{\gamma^u}d\upsilon(\gamma^u)$
we obtain
$$
\mu(T^{-j}U\cap U)\le c_6\omega(j)\mu(U)(j^{-\kappa u_0}+j^{-q}).
$$

Next we  estimate for $j\ge2$ the quantity
\begin{eqnarray*}
\mathbb{P}(Z_0\ge1, Z_j\ge1)
&\le& \sum_{0\le k, \ell <2K+1}\mu(T^{-k}U\cap T^{-\ell-(2K+1)j}U)\\
&=&\sum_{u=(j-1)(2K+1)}^{(j+1)(2K+1)}((2K+1)-|u-j(2K+1)|)\mu(U\cap T^{-u}U)
\end{eqnarray*}
and consequently obtain
\begin{eqnarray*}
\sum_{j=2}^\Delta\mathbb{P}(Z_0\ge1\land Z_j\ge1)
&\le&(2K+1)\sum_{u=2K+1}^{(\Delta+1)(2K+1)}\mu(U\cap T^{-u}U)\\
&\le&c_7K\mu(U)\sum_{u=2K+1}^{(\Delta+1)(2K+1)}\omega(u)(u^{-\kappa u_0}+u^{-q})\\
&\le&c_8K\mu(U)K^{-\sigma}
\end{eqnarray*}
since $\omega(j)=\mathcal{O}(j^{-\kappa'})$, provided
$\sigma=\min\{\kappa u_0, q\}-\kappa'-1$ is less than  $1$.

For the term $j=1$ let $K'<K$ and put $Z'_0=\sum_{i=2K+1-K'}^{2K+1}X_i$
and $Z''_0=Z_0-Z'_0$. Then
$$
\mathbb{P}(Z_0\ge1, Z_1\ge1)
\le\mathbb{P}(Z''_0\ge1, Z_1\ge1)+\mathbb{P}(Z'_0\ge1),
$$
where $\mathbb{P}(Z'_0\ge1)\le K'\mu(U)$. Since by the above estimates
$$
\mathbb{P}(Z''_0\ge1, Z_1\ge1)
\lesssim KK'^{-\sigma}\mu(U)
$$
we conclude that
$$
\mathbb{P}(Z_0\ge1, Z_1\ge1)
\lesssim \mu(U)(K'+KK'^{-\sigma}).
$$

The entire error term is now estimated by
$$
N'\mathcal{R}_2
\le N'\sum_{j=1}^\Delta\mathbb{P}(Z_0\ge1, Z_j\ge1)
\lesssim N'\mu(U)(K^{1-\sigma}+KK'^{-\sigma}+K')
\lesssim t\!\left(K'^{-\sigma}+\frac{K'}K\right)
$$
as $K>K'$.

If $\diam\zeta$ ($\zeta$ $n$-cylinders) and $\mu(\mathcal{G}_n^c)$
decay super polynomially then
$$
\mathcal{R}_2   \lesssim \delta(K')^{u_0}+\delta'(K')+K'/K,
$$
where $\diam \zeta\le\delta(n)$, $\mu(\mathcal{G}_n^c)\le\delta'(n)$ are super polynomial.

In the exponential case ($\delta(n),\delta'(n)=\mathcal{O}(\vartheta^n)$) one has
$$
\mathcal{R}_2 \lesssim \vartheta^{(u_0\land 1)K'}+K'/K.
  $$

\subsection{The total error}\label{total_error} For the total error  we now put $r=\rho^w$
and as above $\Delta=\rho^{-v}$ where $v<d_0$ since $\Delta<\!\!<N$ and $N\ge \rho^{-d_0}$.
Moreover we put $K'=K^\alpha$ for $\alpha<1$.
Then $\mathcal{C}(\Delta)=\mathcal{O}(\Delta^{-p})=\mathcal{O}(\rho^{pv})$
and thus (in the polynomial case)
\begin{eqnarray*}
\left|\mathbb{P}(W=k)-\tilde\nu(\{k\})\right|
&\lesssim& N'\!\left(a^{2K+1}\frac{\mathcal{C}(\Delta)}{r}
+\mu(\mathcal{U}''(r)\setminus\mathcal{U}'(r))\right)+\frac{t}{K^{\sigma'}}
+\frac{t}K(\rho^{v(\kappa\eta-1)-\beta}+\rho^\epsilon)\\
&\lesssim&a^{2K+1}\rho^{vp-w-d_1}+\rho^{w\eta-\beta}+\frac{t}{K^{\sigma'}}
+\frac{t}K(\rho^{v(\kappa\eta-1)-\beta}+\rho^\epsilon)\end{eqnarray*}
as $N'\mu(U)=\frac{s}{2K+1}$ and $s=N'\mathbb{P}(Z^K\ge1)$,
 where we put $\sigma'=\min\{\alpha\sigma,1-\alpha\}>0$.
We can choose $v<d_0$ arbitrarily close to $d_0$ and then require
 $d_0p-w-d_1>0$,  $w\eta-\beta>0$ and $d_0(\kappa\eta-1)-\beta>0$. We can choose $w>\frac\beta\eta$ arbitrarily
 close to $\frac\beta\eta$ and can
 satisfy all requirements if $p>\frac{\frac\beta\eta+d_1}{d_0}$ in the case when $\mathcal{C}$ decays polynomially
 with power $p$, i.e.\ $\mathcal{C}(k)\sim k^{-p}$.

 In the exponential case ($\diam\zeta=\mathcal{O}(\vartheta^n)$ for $n$ cylinders $\zeta$)
we obtain
$$
|\mathbb{P}(W=k)-\tilde\nu(\{k\})|
\lesssim
a^{2K+1}\rho^{s\abs{\log\vartheta}-w-d_1}+\rho^{w\eta-\beta}
+\frac{t}{K^{\sigma'}}+\frac{t}K(\rho^{s\abs{\log\vartheta}-\beta}+\rho^\epsilon).
$$

 \subsection{Convergence to the compound Poisson distribution}
 First observe that for $t>0$ we take $N=t/\mu(U)$ and since by Lemma~\ref{lemma.entry}
 $N'\alpha_1\mu(U)=s$ thius implies that $s=\alpha_1t$.
 We will have to do a double limit of first letting $\rho$ go to zero and then to let
 $K$ go to infinity.
If  $\rho\to0$ then $\mu(U)\to 0$ which implies that $N'\to\infty$ and that the compound
binomial distribution $\tilde\nu$ converges to the
compound  Poisson distribution $\tilde\nu_K$ for the parameters $t\lambda_\ell(K)$.
Thus for every $K$:
 $$
 \mathbb{P}(W=k)\longrightarrow\tilde\nu_K(\{k\})+\mathcal{O}(tK^{-\sigma'}).
$$
Now let $K\to\infty$. Then $\lambda_\ell(K)\to\lambda_\ell$ for all $\ell=1,2,\dots$
and $\tilde\nu_K$ converges to the compound Poisson distribution $\nu$ for the
parameters $s\lambda_\ell=\alpha_1t\lambda_\ell$. Finally we obtain
 $$
 \mathbb{P}(W=k)\longrightarrow\nu(\{k\})
$$
as $\rho\to0$. This concludes the proof of Theorem~\ref{thm1}.
\qed
\section{Examples}\label{examples}

\subsection{A non-uniformly expanding map}
On the torus $\mathbb{T}=[0,1)\times[0,1)$ we consider the affine map $T$
given by the matrix $A=\left(\begin{array}{cc}1&1\\0&a\end{array}\right)$
for some integer $a\ge2$.
This is a partially hyperbolic map since $A$ has one eigenvalue equal to $1$
and is uniformly expanding in the $y$-direction. Horizontal lines are mapped
to horizontal lines and in particular the line $\Gamma=\{(x,0): x\in[0,1)\}$
is an invariant set which entirely consists of fixed points.
Since in estimating the error terms $\mathcal{R}_2$
involves terms of the form $U\cap T^{-n}U$ we only need to consider the
uniformly expanding $y$-direction  when verifying the assumption~(III)(iii).
This means the vertical diameter of $n$-cylinders $\zeta$ contracts exponentially
like $a^{-n}$.

The Lebesgue measure $\mu$ is invariant. To see this notice that $T$ has $a$ inverse
branches whose Jacobians all have determinant $\frac1a$. The neighbourhoods $U$ of $\Gamma$ are
$B_\rho(\Gamma)$. In Assumptions~(IV) and~(V) we thus have  $d_0=d_1=u_0=1$
and in the ``annulus condition''~(VI) one can take $\eta=\beta=1$.

Clearly, the Lebesgue measure is invariant under $A$.
Although this map does not have good decay of correlation we can still apply our
method because the return sets $B_\rho(\Gamma)$ are of very special form
since $A$ maps horizontal lines $y\times [0,1)$ to horizontal lines $y'\times[0,1)$
($y'=ay \mod 1$)  and in vertical direction
is uniformly expanding by factor $a$.

The limiting return times are in the limit compound Poisson distributed where
It is straightforward to determine that
$$
\hat\alpha_{k+1}
=\lim_{\rho\to0}\mu_{B_\rho(\Gamma)}\!\left(T^{-1}B_\rho(\Gamma)\cap T^{-2}B_\rho(\Gamma)
\cap\cdots\cap T^{-k}B_\rho(\Gamma)\right)=\left(\frac1a\right)^k,
$$
$k=1,2,\dots$, since $\mu\!\left(\bigcap_{j=0}^kT^{-j}B_\rho(\Gamma)\right)=a^{-k}\rho$.
 Consequently $\alpha_k=\hat\alpha_k-\hat\alpha_{k+1}=\left(1-\frac1a\right)\left(\frac1a\right)^{k-1}$
 and by Theorem~\ref{theorem.lambda} $\lambda_k=(1-\frac1a)(\frac1a)^{k-1}$, $k=1,2,\dots$,
which shows that the return times to a strip neighbourhood of $\Gamma$
is in the limit P\'olya-Aeppli distributed. (The extremal index is $\alpha_1=1-\hat\alpha_2=1-\frac1a$.)

\subsection{Regenerative processes}

Here we give two examples, one which exhibits some pathology and which was also
recently used in~\cite{AFF19} and another one to show that nearly all
compound Poisson distributions can be achieved.

\subsubsection{Smith example}
To emphasise the regularity condition made in Theorem~\ref{theorem.lambda}
we look at an example by Smith~\cite{Smi88} which was also recently used in~\cite{AFF19', AFF19}
to  exhibit some pathology.

Let $Y_j$ for $j\in\mathbb{Z}$ be i.i.d.\ $\mathbb{N}$-valued random variables
and denote $\gamma_k=\mathbb{P}(Y_j=k)$ its probability density.
For each $k\in\mathbb{N}$, put $p_k=1-\frac1k$ and $q_k=\frac1k$.
Then we define the regenerative process $X_j, j\in\mathbb{Z}$, as follows:
the sequence of $\vec{X}=(\dots,X_{-1},X_0,X_1,X_2,\dots)$ is parsed
into blocks of lengths $\zeta_i\in\mathbb{N}$ so that the sequence
of integers $N_i$ satisfy $N_{i+1}=N_i+\zeta_i$. Then $X_{N_i}=k$ with
probability $\gamma_k$ and $\mathbb{P}(\zeta_i=1)=q_k$ and
$\mathbb{P}(\zeta_i=k)=q_k$. If $\zeta_i=k$ then we put
$X_{N_i+\ell}=k$ for $\ell=1,2,\dots,k$.
That means every time the symbol $k$ is chosen (with probability $\gamma_k$)
then appears a block of only that one symbol with probability $p_k$
or as a block of length $k+1$ of $k$ times repeated symbol $k$ with probability $q_k$.

The sets $U_m=\{\vec{X}: X_0>m\}$ form a nested sequence within the space
$\Omega=\{\vec{X}\}$ which carries the left shift transform $\sigma:\Omega\circlearrowleft$.
Moreover there exists a $\sigma$-invariant probability measure $\mu$
for which $\mu(\{k\})= \gamma_k$.
To find $\hat\alpha_k(L)$ for (large) $L$ we let $m>\!\!L$. For $\vec{X}\in\Omega$,
let $i$ be so that $N_i(\vec{X})\le0<N_{i+1}(\vec{X})$. Then $\zeta_i=N_{i+1}-N_i$ is the length
of the block containing $X_0$.
Let $\varepsilon>0$, then for all $k$ large enough we have
$$
\mathbb{P}(\zeta_i=1|X_0=k)=\frac{p_k}{p_k+(k+1)q_k}\in\left(\frac12-\varepsilon,\frac12\right)
$$
as $\mathbb{E}(\zeta_i|X_0)=2$ for all $k$. Similarly
$$
\mathbb{P}(\zeta_i=k+1|X_0=k)=\frac{(k+1)q_k}{p_k+(k+1)q_k}\in\left(\frac12,\frac12+\varepsilon\right)
$$
for all $k$ large enough. In particular, for all $m$ large enough,
$$
\mathbb{P}(\zeta_i=1|U_m)\in\left(\frac12-\varepsilon,\frac12\right),
\hspace{2cm}
\mathbb{P}(\zeta_i>1|U_m)\in\left(\frac12,\frac12+\varepsilon\right).
$$
Therefore ($\zeta_1>1$ here means $\zeta_i>m$)
\begin{eqnarray*}
\left|\mathbb{P}_{U_m}\!\left(\tau_{U_m}^{k-1}>L\right)-\frac12\right|
&\le&\mathbb{P}(\zeta_i=1|U_m)\mu(U_m)
+\!\left|\mathbb{P}(\zeta_i=1|U_m)-\frac12\right|\\
&&+\mathbb{P}(\zeta_i>1|U_m)\frac{L}m
+\!\left|\mathbb{P}(\zeta_i>1|U_m)-\frac12\right|\\
&\le&4\varepsilon
\end{eqnarray*}
for all $m$ large enough so that in particular also $L/m<\varepsilon$ and $\mu(U_m)<\varepsilon$.
The first term on the RHS comes from the events that re-enter $U_m$  after exiting
and the third term accounts for the probability that the block of length $\zeta_i$ does not
cover the entire interval $(0,L]$.
Consequently
$$
\hat\alpha_k(L)=\lim_{m\to\infty}\mathbb{P}_{U_m}\!\left(\tau_{U_m}^{k-1}>L\right)=\frac12
$$
for all $k\ge2$ and for all $L$ (trivially $\hat\alpha_1=1$). Consequently $\hat\alpha_k=\frac12$ for all $k=1, 2,\dots$.
Moreover we also obtain that $\alpha_1=\frac12$ and $\alpha_k=0$ for all $k\ge2$.

Since the condition of Theorem~\ref{theorem.lambda} is not satisfied, we cannot
use it to obtain the probabilities $\lambda_k$ for the $k$-clusters.
We can however proceed more directly by noting that
$$
\mathbb{P}(Z^L>0)=(2L+1)\mu(U_m)(1-\mathcal{O}^*(1/m))-\mathcal{O}(\mu(U_m)^2g(L,\mu(U_m))),
$$
where  $Z^L=\sum_{j=-L}^{L}\ind_{U_m}\circ \sigma^j$ and $g(L,\mu(U_m))$
 is a function which is bounded and stays bounded as
$\mu(U_m)\to0$ ($\mathcal{O}^*$ expresses that the implied constant is $1$, i.e.\
$x=\mathcal{O}^*(\epsilon)$ if $|x|<\epsilon$)  Similarly we get that
$$
\mathbb{P}(Z^L=1)=(2L+1)\mu(U_m)+\mathcal{O}(\mu(U_m)^2g'(L,\mu(U_m)))
$$
where $g'$ is like $g$. Also
$$
\mathbb{P}(Z^L>1)=\mathcal{O}(\mu(U_m))
$$
where the implied constant depends on $L$.
Consequently
$$
\lambda_k(L)=\lim_{n\to\infty}\frac{\mathbb{P}(Z^L=k)}{\mathbb{P}(Z^L>0)}
=\mathcal{O}(1/L)\to0
$$
as $L\to\infty$ and therefore $\lambda_k=0$ for all $k\ge2$. For $k=1$ we obtain
$$
\lambda_1(L)=\lim_{n\to\infty}\frac{\mathbb{P}(Z^L=1)}{\mathbb{P}(Z^L>0)}
=1.
$$
This does not square with the statement of Theorem~\ref{theorem.lambda}
since the we have masses that are wandering off to infinity.


\subsubsection{Arbitrary parameters} We use an example which is similar to
Smith's to show that any sequence of parameters $\lambda_k$ can be realised
as long as the expected value is finite.
As above let $Y$ be an $\mathbb{N}$-valued random variable with
probability distribution $\gamma_k=\mathbb{P}(Y=k)$.
Let $\lambda_k, k=1,2,\dots$, be a sequence of parameter values
so that $\sum_{k=1}^\infty\lambda_k=1$ and $\sum_{k=1}^\infty k\lambda_k<\infty$.
As above we define the regenerative process $X_j, j\in\mathbb{Z}$ by parsing
the sequence of $\vec{X}=(\dots,X_{-1},X_0,X_1,\dots)$
into blocks of lengths $\zeta_i\in\mathbb{N}$ so that the sequence
of integers $N_i$ which indicates the heads of runs satisfy $N_{i+1}=N_i+\zeta_i$.
Then $X_{N_i}=k$ with probability $\gamma_k$ and $\mathbb{P}(\zeta_i=j)=\lambda_j$.
That means that blocks of the symbol $k$ which are of length $j$ are chosen with
the given probability $\lambda_j$. Put $\Omega=\{\vec{X}\}$.

As before, let $U_m=\{\vec{X}\in \Omega: X_0>m\}$.
For $\vec{X}\in\Omega$ let $i$ be so that
$N_i\le0<N_{i+1}$. Then $X_0$ belongs to a block of length $\zeta_i=N_{i+1}-N_i$.
This implies
$$
\mathbb{P}(\zeta_i=\ell)=\frac{\ell\lambda_\ell}{\sum_{s=1}^\infty s\lambda_s}
$$
Also
$\mathbb{P}(X_0=X_1=\cdots X_{k-1}\not=X_k|\zeta_i=\ell)=1/\ell$
and consequently for $k<m$:
$$
\alpha_k(L,U_m)=\frac{\sum_{\ell=k}^\infty\lambda_\ell}{\sum_{s=1}^\infty s\lambda_s}
+\mathcal{O}(L\mu(U_m))
$$
where the error terms expresses the likelyhood for entering the set $U_m$ after
the $\zeta_i$-block of being inside $U_m$. Taking a limit $m\to \infty$ we obtain
$$
\alpha_k=\lim_{L\to\infty}\alpha_k(L)=\frac{\sum_{\ell=k}^\infty\lambda_\ell}{\sum_{s=1}^\infty s\lambda_s}.
$$
In particular if $k=1$ we get
$\alpha_1=1/\sum_{s=1}^\infty s\lambda_s=1/\mathbb{E}(\zeta_i)$ as $\sum_{\ell=1}^\infty\lambda_\ell=1$.
This is the relation to be expected in general, where the extremal index $\alpha_1$ is
the reciprocal of the expected value of the cluster length.

\subsection{Periodic points}\label{section.periodic.points}

 For a set $U\subset\Omega$ we write $\tau(U)=\inf_{y\in U}\tau_U(y)$ for
 the {\em period} of $U$. In other words,  $U\cap T^{-j}U=\varnothing$ for
 $j=1,\dots,\tau(U)-1$ and $U\cap T^{-\tau(U)}U\not=\varnothing$.
 Let us now consider a sequence of nested sets $U_n\subset \Omega$
 so that $U_{n+1}\subset U_n\forall n$ and $\bigcap_nU_n=\{x\}$ a single point
 $x$. Then we have the following simple result which is independent of the
 topology or an invariant measure on $\Omega$.

 \begin{lem} Let $U_n\subset\Omega$ be so that $U_{n+1}\subset U_n\forall n$
 and $\bigcap_nU_n=\{x\}$ for some $x\in\Omega$.
 Then the sequence $\tau(U_n)$, $n=1,2,\dots$ is bounded if and only if
 $x$ is a periodic point.
 \end{lem}

 \noindent {\bf Proof.} If we put $\tau_n=\tau(U_n)$ then $\tau_{n+1}\ge \tau_n$
 for all $n$. Thus either $\tau_n\to\infty$ or $\tau_n$ has a finite limit $\tau_\infty$.
 Assume $\tau_n\to\tau_\infty<\infty$. Then $\tau_n=\tau_\infty$ for all $n\ge N$, for some $N$,
 and thus $U_n\cap T^{-\tau_\infty}U_n\not=\varnothing$ for all $n\ge N$.
 Since the intersections $U_n\cap T^{-\tau_\infty}U_n$ are nested, i.e.\
 $U_{n+1}\cap T^{-\tau_\infty}U_{n+1}\subset U_n\cap T^{-\tau_\infty}U_n$
we get
 $$
\varnothing\not=\bigcap_{n\ge N}(U_n\cap T^{-\tau_\infty}U_n)
 =\bigcap_{n\ge N}U_n\cap\bigcap_{n\ge N} T^{-\tau_\infty}U_n
 =\{x\}\cap\{T^{-\tau_\infty}x\}
 $$
 which implies that $x=T^{\tau_\infty}x$ is a periodic point.
 Conversely, if $x $ is periodic then clearly the $\tau_n$ are bounded by its period.
 \qed

 Let us now compute the values $\lambda_\ell$. Assume $x$ is a periodic point
 with minimal period $m$, then
 $p_i^\ell=\mu_{U_n}(\tau^{\ell-1}_{U_n}=i)=0$ for $i<m$ and $m=\tau(U_n)$
 if $n$ is large enough. For $n$ large enough one has $\tau(U_n)=\tau_\infty=m$
 and therefore $U_n\cap\{\tau_{U_n}=m\}=U_n\cap T^{-m}U_n$.

 Assume the limit
  $p=p^2_m=\lim_{n\to\infty}\frac{\mu(U_{n}\cap T^{-m}U_n)}{\mu(U_n)}$
  exists, then one also has more generally
 $$
 p^\ell_{(\ell-1)m}=\lim_{n\to\infty}\frac{\mu(\bigcap_{j=1}^{\ell-1}T^{-jm}U_{n})}{\mu(U_n)}=p^{\ell-1}.
 $$

 All other values of $p^\ell_i$ are zero, that is $p^\ell_i=0$ if $i\not=(\ell-1)m$.
 Thus $\hat\alpha_\ell=p^\ell_{(\ell-1)m}=p^{\ell-1}$
  and consequently
 $$
 \alpha_\ell=\hat\alpha_\ell-\hat\alpha_{\ell+1}=(1-p)p^{\ell-1}
 $$
  which is a geometric distribution.
 This implies that the random variable $W$ is in the limit P\'olya-Aeppli distributed
 with the parameters $\lambda_k=(1-p)p^{k-1}$.

In particular  the extremal index here is $\alpha_1=1-\hat\alpha_2=1-p$.
\begin{remark}
As we anticipated in Remark \ref{PR},  the extremal index can be explicitly evaluated in some cases. Two examples are:
 For one-dimensional maps $T$  of Rychlik type with potential $\phi$ (with zero pressure) and equilibrium state
  $\mu_{\phi}$, if $x$ is a periodic point of prime period $m$, then we get Pitskel's value
  $\alpha_1=1-e^{\sum_{k=0}^{m-1}\phi(T^kx)}$, see~\cite{Pit91, HV09,  FFT13}.
  For piecewise multidimensional expanding maps $T$ considered in~\cite{Sau00},
  if $\xi$ is again  a periodic point of prime period $p$, then $\alpha_1=1-|\det D(T^{-p})(\xi)|,$
   see Corollary 4 in~\cite{FFT13} and also~\cite{AB10}.
 \end{remark}

 If $x$ is a non-periodic point, then $\tau_n=\tau(U_n)\to\infty$ as $n\to\infty$
 which implies that $\mathbb{P}(\tau_{U_n}\le L)=0$ for all $n$ large enough (i.e.\ when
 $\tau_n>L$), and therefore for all $k\ge2$ $\hat\alpha_k(L)\le\lim_n\mathbb{P}_{U_n}(\tau_{U_n}\le L)=0$.
 That is $\alpha_1=\lambda_1=1$ and $\alpha_k=\lambda_k=0\forall k\ge2$.
 The limiting return times distribution is therefore a regular Poisson distribution.

\section{Coupled map lattice} \label{synchronisation}
Let $T$ be a piecewise continuous map on the unit interval $[0,1]$. We want to consider
a map $\hat{T}$ on $\Omega=[0,1]^n$ for some integer $n$ which is given by
\begin{equation}\label{TMn}
\T(\vec{x})_i=(1-\gamma)T(x_i)+\gamma \sum_{j=1}^n M_{i,j}T(x_j)\qquad \forall\, i=1, 2,\dots,n,
\end{equation}
for $\vec{x}\in \Omega$, where $M$ is an $n\times n$ stochastic matrix and  $\gamma\in[0,1]$ is a
coupling constant. For $\gamma=0$ we just get the product of $n$ copies of $T$.
We assume that $T$ is a piece-wise  expanding map of the unit interval onto itself,  with a finite number
of branches, say $q,$ and which $T$ is assumed to be of class $C^2$ on the interiors of the domains of
injectivity $I_1,\dots, I_q,$ and extended by continuity to the boundaries. Whenever the coupling constant $\gamma=0$ the map $\T$ is the direct product of $T$ with itself; therefore $\T$ could be seen as a coupled map lattice (CML).
Let us denote by $U_k, k=1,\dots, q^n,$ the domains of local injectivity of $\hat{T}$.  By the previous
assumptions on $T$, there exist open sets $W_k\supset U_k$ such that $\T|_{W_k}$ is a $C^2$ diffeomorphism (on the image).  We will require that
$$
s:=\sup_k\sup_{\vec{x}\in \T (W_k)}||D\T|_{W_k}^{-1}(\vec{x})||< b<1,
$$
where
$b:=\sup_i\sup_{x\in T(A_i)}|DT|_{I_i}^{-1}(x)|,$ and $||\cdot||$ stands for the euclidean norm. We will write {\em dist} for the distance with respect to this norm. We will suppose that the map $\T$ preserve an absolutely continuous invariant measure $\mu$ which is moreover mixing.
Recall that  $\text{osc}(h,A):=\text{Esup}_{\vec{x}\in A}h(\ox)-\text{Einf}_{\ox\in A}h(\vec{x})$ for any
measurable set $A$: see the proof of Lemma \ref{L4} for a more detailed definition.  Finally  $\Le$ is the Lebsegue measure on $\Omega$.

Let
\begin{equation}\label{set}
S_{\nu}:=\{\vec{x}\in I^n: |x_i-x_j|\le \nu \forall i,j\}
\end{equation}
be the $\nu$-neighbourhood of the diagonal $\Delta$.
Then $\hat\alpha_{k+1}(L,S_\nu)=\mathbb{P}(\tau^k_{S_\nu}\le L|S_\nu)$.
The value $\hat\alpha_{k+1}$ is the limiting probability of staying in the neighborhood of the diagonal until  time
$k$ and as the strip $S_{\nu}$  collapses to the diagonal $\Delta$.

\begin{thrm} \label{theorem.synchronisation}
Let $\T:\Omega\to\Omega$ be a coupled map lattice over the uniformly expanding map $T:[0,1]\circlearrowleft$
 and assume
that the hypersurfaces of discontinuities are piecewise $C^{1+\alpha}$ and intersections
with the diagonal $\Delta$ are transversal.
Moreover suppose
the stochastic matrix $M$ has constant columns, that is $M_{i,j}=p_j$ for a
probability vector $\vec{p}=(p_1,\dots,p_n)$ and assume the map $\hat{T}$ satisfies Assumption (I) for any $\gamma\in [0,1].$

 Finally suppose  that the density $h$ of the invariant absolutely continuous probability measures
 $\mu$ satisfies
 $$
  \sup_{0<\eps\le \eps_0}\frac{1}{\eps}\int_0^1 \text{osc}(h, B_{\eps}((x)^n)\, dx<\infty,
    $$
    where  $(x)^n\in\Delta$ is the point on the diagonal  all of whose coordinates are equal to $x\in[0,1]$.

    Then
$$
 \hat{\alpha}_{k+1}
 =\lim_{L\to\infty}\lim_{\nu\to0}\hat\alpha_{k+1}(L,S_\nu)
 =\frac1{(1-\g)^{k(n-1)} \int_Ih((x)^n) \,dx}\int_I \frac{h((x)^n)}{|DT^k(x)|^{n-1}}\,dx
 $$
 and the limiting return times to the diagonal $\Delta$ are compound Poisson distributed
 with parameters $t\lambda_k$ where
 $\lambda_k=\frac1{1-\hat\alpha_2}(\hat\alpha_{k-1}-2\hat\alpha_k+\hat\alpha_{k+1})$ and $t>0$ is real.
\end{thrm}

\begin{remark}
If $|DT|$ is constant, as for instance for the doubling map, then we obtain
$\hat\alpha_{k+1}=\left((1-\gamma)|DT|\right)^{-k(n-1)}$.
This implies that the probabilities
$\alpha_k=\hat\alpha_k-\hat\alpha_{k+1}$ are geometric and, by extension,  also
$$
\lambda_k=\left((1-\gamma)|DT|\right)^{-(k-1)(n-1)}\left(1-\left((1-\gamma)|DT|\right)^{-(n-1)}\right).
$$
 This means that the cluster sizes are geometrically distributed and
therefore the limiting return times distribution is P\'olya-Aeppli.
If $|DT|$ is non-constant then in general we cannot expect the probabilities $\alpha_k$ and $\lambda_k$
to be geometric, which implies that in the generic case, the limiting return times
distribution is not P\'olya-Aeppli. This should clarify a remark made in~\cite{FGGV}, Section 6.
\end{remark}

\begin{remark}
We now give an example of a map verifying Assumption~(I). Suppose the map $T$ is defined on
the unit circle as $T(x)=ax\mod 1$, with $a\in \mathbb{Z}$. Then, by using the quantities
$M$ and $B$ introduced in the proof of Theorem \ref{theorem.synchronisation},
it is easy to see that $\hat{T}^k(\vec{x})=B^k(a^k x_1\mod 1, \cdots, a^k x_n\mod 1)^T$,
and therefore the images of the $k$-cylinders will be the whole space.
\end{remark}

 For $k=1$ a proof already appeared in \cite{FGGV}; the proof which we give here is
considerably simpler and easily adaptable to other coupled map lattices. In particular,
instead of using the transfer operator to determine the measure of $S_\nu^k$ (below)
we use the tangent map of the coupled map in the neighbourhood of the diagonal.

Let us put
$$
S_\nu^k
=\bigcap_{\ell=0}^k\hat{T}^{-\ell}S_\nu
=\left\{\vec{x}\in \Omega: |(\T^\ell(\vec{x}))_i-(\T^\ell(\vec{x}))_j|< \nu, \ell=0,\dots,k\right\}
$$
as the set of points in $S_\nu$ which for $k-1$ iterates of $\hat{T}$ stay in the $S_\nu$-neighbourhood
of the diagonal $\Delta$. We proceed in two steps.

\begin{lem}\label{L4}
Under the assumption of Theorem~\ref{theorem.synchronisation} we get
$$
 \hat{\beta}_{k+1}
 :=\lim_{\nu\to0} \frac{\mu(S_\nu^k)}{\mu(S_{\nu})}
 =\frac1{(1-\g)^{k(n-1)} \int_Ih((x)^n) \,dx}\int_I \frac{h((x)^n)}{|DT^k(x)|^{n-1}}\,dx.
 $$
\end{lem}

\begin{proof}
The density $h$ of $\mu$ is the (unique)  eigenfunction of the transfer operator acting on the space of quasi-H\"older functions, see \cite{Kel}  and  especially \cite{Sau00}.
For all functions $h$ on $\Omega $ we define a  semi-norm $|h|_\alpha$ which, given two real numbers
$\eps_0>0$ and $0<\alpha\le 1$, writes
$$
|h|_{\alpha}:=\sup_{0<\eps\le \eps_0}\frac{1}{\eps^{_\alpha}}\int \text{osc}(h, B_{\eps}(\vec{x})) \,d\Le(\vec{x}).
$$
 We say that $h\in V_{\a}(\Omega)$ if $|h|_{\a}<\infty$. Although the value of $|h|_{\a}$ depends on $\eps_0,$
  the space $ V_{\a}(\Omega)$ does not. Moreover the value of $\eps_0$ can be  chosen in order to satisfy a few geometric constraints, like distortion, and to guarantee the Lasota-Yorke inequality on the
 Banach space $\mathcal{B}=(V_{\a}(\Omega), ||\cdot||_{\a})$, where the norm $||\cdot||_{\a}$ is defined as
 $||h||_{\a}:=|h|_{\a}+||h||_1$.
It has been shown~ \cite{Sau00} that $\mathcal{B}$ can continuously be injected into $\mathscr{L}^{\infty}$
since $||h||_{\infty}\le C_H ||h||_{\a}$, where $C_H=\frac{\max(1,\eps_0^{\a})}{Y_n \eps_0^n}$, being $Y_n$
the volume of the unit ball in $\mathbb{R}^n.$ The density in the neighborhood
 of the diagonal $\Delta$ is controlled by the assumption
$$
    h_D:=\sup_{0<\eps\le \eps_0}\frac{1}{\eps}\int_0^1 \text{osc}(h, B_{\eps}((x)^n)\, dx<\infty,
    $$
    where  $(x)^n\in\Delta$.
This means that we compute the oscillation in balls moving along the diagonal. By decreasing the radius $\eps$ the oscillation decreases; this plus Fatou Lemma implies that
$$
\lim_{\eps\rightarrow 0} \text{osc}(h, B_{\eps}((x)^n)=0,
$$
for Lebesgue almost all $x\in [0,1],$ which in turns implies that {\em $h$ is almost everywhere continuous
along the diagonal}.
Consequently, if $x_1$ is chosen almost everywhere in $[0,1]$ and the vector $(y_2, \dots, y_n)$ is chosen almost everywhere (with respect to the Lebesgue measure on $\mathbb{R}^{n-1}$) in a ball of radius $\nu<\eps_0$ around the point $(x_1)^n,$ we have $|h(x_1, y_2, \dots, y_n)-h((x_1)^n)|\le \text{osc}(h, B_{\nu}((x_1)^n)))$  and therefore
$$
\int  |h(x_1, y_2, \dots, y_n)-h((x_1)^n)|\,dx_1\le \int \  \text{osc}(h, B_{\eps}((x_1)^n))\,dx_1\le \nu \, h_D,
$$
which goes to $0$ when $\nu$ tends to zero.

For the neighbourhood $S_{\nu}:=\{\vec{x}\in I^n: |x_i-x_j|\le \nu \forall i,j\}$
of the diagonal $\Delta$,
we now want to compute the limit
$\hat{\beta}_{k+1}=\lim_{\nu\to0} \frac{\mu(S_\nu^k)}{\mu(S_{\nu})}$,
where as before $S_\nu^k=\bigcap_{j=0}^kS_\nu$,
which measures the limiting probability of staying in the neighborhood of the diagonal until  time
$k$ and as the strip $S_{\nu}$  collapses to the diagonal $\Delta$.
We begin to observe that the derivative $D\T$ has the form
$$
D\T=\left((1-\gamma)\id+\gamma M\right)D\mathbb{T},
$$
or $D\T=B\cdot D\mathbb{T}$, where $D\mathbb{T}( {T})$ is the diagonal $n\times n$ matrix
with diagonal entries $DT(x_1), DT(x_2),\dots, DT(x_n)$ and $B=(1-\gamma)\id+\gamma M$.

Let $\vec{u}=n^{-\frac12}(1, 1, \dots, 1)$ be the unit vector that spans the diagonal $\Delta$.
For a point $\vec{x}\in\Omega$ put $\vec{v}$ for the vector in $\mathbb{R}^n$ with
components $v_j=x_j-x_0$ where $x_0\in[0,1]$ is arbitrary. Then
$$
d(\vec{x},\Delta)=\left(|\vec{v}|^2-(\vec{v}\cdot\vec{u})^2\right)^\frac12
$$
is distance of $\vec{x}$ from the diagonal.

For $x_0\in[0,1]$ denote by $(x_0)^n\in\Delta$  the point on the diagonal  all of whose coordinates
are equal to $x_0$. Notice that $\T$ leaves the diagonal invariant as
$\T((x_0)^n)=(T(x_0))^n$.
If $\vec{v}$ and $(x_0)^n$ lie in the same region of continuity of
$\T^\ell$ then
$$
d(\T^\ell((x_0)^n),\T^\ell(\vec{x}))=D\T^\ell((x_0)^n)\vec{v}+\mathcal{O}(|\vec{v}|^2),
$$
where as before $\vec{v}=\vec{x}-(x_0)^n$ and  $D\T^\ell((x_0)^n)=DT^\ell(x_0)B^n$.
Consequently
$$
d(\T^\ell(\vec{x}),\Delta)=\left(|\vec{v}(\ell)|^2-(\vec{v}(\ell)\cdot\vec{u})^2\right)^\frac12,
$$
where ($\vec{v}=\vec{v}(0)$)
$$
\vec{v}(\ell)=\T^\ell(\vec{x})-\T^\ell((x_0)^n)
=DT^\ell(x_0)B^\ell\vec{v}+\mathcal{O}(|\vec{v}|^2).
$$
Using the linearisation of $\T$, the set $S_\nu^k$ is approximated by
$$
\tilde{S}_\nu^k=\left\{\vec{x}\in\Omega:
DT^\ell(x_0)\left(\left|B^\ell\vec{v}\right|^2-\left(B^\ell\vec{v}\cdot\vec{u}\right)^2\right)^\frac12\le\nu,
\vec{v}=\vec{x}-(x_0)^n\right\}.
$$

Let us consider the special case when $M$ has constant columns, that is
$M_{i,j}=p_j$, where $\vec{p}=(p_1,p_2,\dots,p_n)$ is a probability vector.
Then $M^\ell=M$ for $\ell=1,2, \dots$ and
$$
B^\ell=\left((1-\gamma)\id+\gamma M\right)^\ell
=(1-\gamma)^\ell\id+\left(1-(1-\gamma)^\ell\right)M
$$
which yields
$$
B^\ell\vec{v}=(1-\gamma)^\ell\vec{v}+\left(1-(1-\gamma)^\ell\right)\sqrt{n}(\vec{v}\cdot\vec{p})\vec{u},
$$
as $M\vec{v}=\sqrt{n}(\vec{v}\cdot\vec{p})\vec{u}$.
Thus
$$
|B^\ell\vec{v}|^2-\left((B^\ell\vec{v})\cdot\vec{u}\right)^2
=(1-\gamma)^{2\ell}\left(|\vec{v}|^2-V^2\right),
$$
 where $V=\sum_{j=1}^nv_j$.
 If we can choose $x_0=\frac1n\sum_{j=1}^nx_j$ then $V=0$ and
the distance $\left(|B^\ell\vec{v}|^2-\left((B^\ell\vec{v})\cdot\vec{u}\right)^2\right)^\frac12$ is equal to
$(1-\gamma)^\ell|\vec{v}|$. For this the points $\vec{x}$ and $(x_0)^n$ have to lie in the same
connected partition element of continuity for $\T$.

Since $M_{i,j}=p_j\forall i,j$ and if we choose $x_0=\frac1n\sum_{j=1}^nx_j$ we obtain
$B^\ell\vec{v}\cdot\vec{u}=0$ and
$$
\left|B^\ell\vec{v}\right|
=(1-\gamma)^{\ell}|\vec{v}|
=(1-\gamma)^{\ell}d(\vec{x},\Delta).
$$
Consequently
$$
|\vec{v}(\ell)|=(1-\gamma)^\ell |DT^\ell(x_0)|\cdot|\vec{v}|
$$
and $d(\hat{T}^\ell\vec{x},\Delta)=(1-\gamma)^\ell |DT^\ell(x_0)|d(\vec{x},\Delta)+o(d(\vec{v},\Delta))$.
Therefore in linear approximation
$$
\tilde{S}_\nu^k=\left\{\vec{x}\in\Omega:
 d(\vec{x},\Delta)\le\frac{\nu}{DT^\ell(x_1)(1-\gamma)^\ell}, \ell=0,1,\dots,k\right\}
$$
and since $T$ is expanding only the term $\ell=k$ is relevant.

Denote by $\mathcal{D}^k$ the set of discontinuity points for $\hat{T}^\ell$ for $\ell=1,\dots,k$.
We assume that  $\mathcal{D}^k$ is a union of piecewise smooth hyper surfaces which
intersect the diagonal $\Delta$ transversally. Then $\mathcal{D}^k\cap\Delta=\{(y_1)^n,(y_2)^n,\dots, (y_m)^n\}$
consists of finitely many points $(y_j)^n\in\Delta$. For each $j$ denote by
$\varphi_j=\angle(\Delta,\mathcal{D}^k)$ the angle between $\Delta$ and $\mathcal{D}^k$ at the point
of intersection $(y_j)^n$. Clearly the angles $\varphi_j$ are bounded away
from $0$ and we can put $r=2\nu(\cot\varphi+1)$ where $\varphi=\min_j\varphi_j$.
If we put $\Delta_\nu^k=\Delta\setminus\bigcup_jB_r((y_j)^n)$ then
for all $\nu$ small enough $B_\nu(\Delta_\nu^k)\cap\mathcal{D}^k=\varnothing$.

In order to compute $\mu(S_\nu^k)$ and $\mu(S_\nu)$ put
$$
S_\nu^k(x_1)=\left\{(x_2,x_3,\dots,x_n)\in [0,1]^{n-1}: |T^\ell(x_1)-T^\ell(x_j)|\le\nu, j=2,\dots,n,\;\;\ell=1,\dots,k\right\}.
$$
Then $S_\nu^\ell=\bigcup_{x_1\in[0,1]}\{x_1\}\times S_\nu^\ell(x_1)$ for $\ell=1,\dots,k$.
In the same fashion we can look at the linear appproximation and put
$$
\tilde{S}_\nu^k(x_1)
=\left\{(x_2,x_3,\dots,x_n)\in [0,1]^{n-1}: |DT^\ell(x_1)|\cdot|x_1-x_j|\le\nu, j=2,\dots,n,\;\;\ell=1,\dots,k\right\}.
$$
By the $C^2$-regularity of the maps one obtains
$$
\int_{S_\nu^k(x_1)}\,dx_2\cdots dx_n
=(1+\mathcal{O}(\nu))\int_{\tilde{S}_\nu^k(x_1)}\,dx_2\cdots dx_n
=(1+\mathcal{O}(\nu))\left(\frac{2\nu}{(1-\gamma)|T^k(x_1)|}\right)^{n-1}
$$

As we concluded above, we obtain by regularity of the density $h$ that
$$
\mu(S_\nu)=\int_{S_\nu}h(\vec{x})\,d\vec{x}
=(1+o(1))\int_{S_\nu}h((x_1)^n)\,d\vec{x}
$$
where the second integral is
$$
\int_{S_\nu}h((x_1)^n)\,d\vec{x}
=\int_{[0,1]}\int_{S_\nu^0(x_1)}h((x_1)^n)\,dx_2\cdots dx_n \,dx_1
=\int_{[0,1]}h((x_1)^n)(2\nu)^{n-1} \,dx_1
$$
as $\int_{S_\nu^0(x_1)}dx_2\cdots dx_n=(2\nu)^{n-1} $.

Similarly we obtain
\begin{eqnarray*}
\mu(S_\nu^k)&=&(1+o(1))\int_{S_\nu^k}h((x_1)^n)\,d\vec{x}\\
&=&(1+o(1))\int_{[0,1]}\int_{S_\nu^k(x_1)}h((x_1)^n)\,dx_2\cdots dx_n \,dx_1\\
&=&(1+o(1))\int_{[0,1]}h((x_1)^n)\left(\frac{2\nu}{(1-\gamma)|T^k(x_1)|}\right)^{n-1} \,dx_1.
\end{eqnarray*}

Finally
 $$
 \hat{\beta}_{k+1}=\lim_{\nu\to0}\frac{\mu(S_\nu^k)}{\mu(S_\nu)}
 =\frac{\int_I \frac{h((x)^n)}{|DT^k(x)|^{n-1}}\,dx}{(1-\g)^{k(n-1)} \int_Ih((x)^n) \,dx}.
 $$
 \end{proof}

The second ingredient to Theorem~\ref{theorem.synchronisation} is the following lemmma which
establishes that all returns to $S_\nu$ within a cluster are of first order which makes
$\Delta$ look like a fixed point. That is $\hat\beta_k=\hat\alpha_K$:

\begin{lem} \label{lemma.equality}
Under the assumptions of Theorem~\ref{theorem.synchronisation}
$$
\hat{\alpha}_{k+1}=\lim_{\nu\to0} \frac{\mu(S_\nu^k)}{\mu(S_{\nu})}.
$$
\end{lem}

\begin{proof} We follow the proof of Proposition 5.3 in \cite{FGGV} adapted to our setting. We begin to consider again  the set $\Delta_\nu^k=\Delta\setminus\bigcup_jB_r((y_j)^n).$ The $\nu$-neighborhood of $\Delta,$ $\Delta_\nu^k$,  will be a subset of $S_{\nu}$ with empty intersection with the discontinuity surfaces $\mathcal{D}^k$ of the maps $\hat{T}^\ell$ for $\ell=1,\dots,k$. We put $G_1(\nu):=\bigcup_jB_r((y_j)^n).$ For reasons which will be clear in a moment, we now remove from the $\nu$-neighborhood of the diagonal, another set. Consider the intersection points $\{(z_1)^n,(z_2)^n,\dots, (z_l)^n\}$ of $\Delta$ with the images of the discontinuity surfaces $\mathcal{D}$ of $\hat{T}$ only, and as we did previously we introduce the set $G_2(\nu):=\bigcup_iB_{2r}((z_i)^n),$ where we double the radius to allow an upcoming  construction. Notice that with the choice of $r$ given above, we have that $\mu(G_1(\nu))=o(\mu(S_{\nu}),$ and $\mu(G_2(\nu))=o(\mu(S_{\nu}),$ when $\nu \rightarrow 0.$

Let us  take a point $x\in \Delta_\nu^k$ and  a neighborhood $\mathcal{O}(x)$ such that
$\mathcal{O}(x)\cap \Lambda\neq \varnothing,$  and
$\mathcal{O}(x)\cap (\mathcal{D}^k\cup \hat{T}^{-1}(\mathcal{D}^k)\cup \cdots \hat{T}^{-k}(\mathcal{D}^k))=\varnothing$.
With these assumptions, $\hat{T}^\ell$ for $\ell=1,\dots,k$
are open maps on $\mathcal{O}(x)$. In particular, $\hat{T}^k(\mathcal{O}(x))$
will be included in the interior of one of the $U_l$ and
 it will intersect $\Delta$ by the forward invariance of the latter. We now suppose that $\hat{T}^k(x)$ is in $S_{\nu}$ and we  prove that $\hat{T}^{k-1}(x)$ is in $S_{\nu}$ too. Let us call $D_*$ the domain of the function  $\hat{T}^{-1}_*$, namely the inverse branch of the map sending $\hat{T}^{k-1}(x)$ to  $\hat{T}^k(x). $ If the distance between $\hat{T}^k(x)$ and any point $z\in \hat{T}^k(\mathcal{O}(x))\cap \Delta,$  such that the segment $[ \hat{T}^k(x), z]$ is included in $D_*,$ is less than $\nu$, we have done since $\text{dist}(\hat{T}^{-1}_*(z),
 \hat{T}^{-1}_*(\hat{T}^k(x))=\text{dist}(\tilde{z},\hat{T^{k-1}}(x))\le \lambda \nu,$  where $\tilde{z}=\hat{T}^{-1}_*(z)\in \Delta.$ Notice that such a point $z\in \Delta$ should not be necessarily in $\hat{T}^k(\mathcal{O}(x)),$ provided the segment $[ \hat{T}^k(x), z]\in D_*$ and $\text{dist}(z,\hat{T^k}(x))\le \ \nu.$ What could prevent the latter conditions to happen is the presence of the boundaries of the domains of definition of the preimages of $\hat{T}$, which are the images of $\mathcal{D}.$
We should therefore avoid that $\hat{T}^k(x)$ lands in the set $G_2(\nu),$  which, with the choice of doubling the radius $r$, is enough large to allow  the point $\hat{T^k}(x)\in G_2(\nu)^c$ to be joined to $\Delta$ with a segment included in $D_*.$  We have therefore  to discard those points  $x\in S_{\nu}$ which are in  $\hat{T}^{-k}G_2({\nu})$ and, by invariance, the measure of those point is bounded from above by $\mu(G_2({\nu})).$  We now iterate backward the process to guarantee that $T^{k-2}(x)$ is in $S_{\nu}$ too. At this regard we must avoid again that  $T^{k-2}(x)\in G_2({\nu}),$ which means we have to remove a new portion of points of measure $\mu(G_2({\nu}))$  from $S_{\nu};$ at the end we will have $k$ times of this measure of order $o(\nu).$ In conclusion
 the points which are not in $\bigcup_{l=2}^k\hat{T}^{-l}G_2(\nu)\cap S_{\nu}\cap G_1(\nu) $ gives zero contribution  to the quantity $\mu(S^k_{\nu})$, while the measure of the remaining points divided by $\mu(S_{\nu})$ goes to zero for $\nu$ tending to zero.
\end{proof}

\begin{proof}[Proof of Theorem~\ref{theorem.synchronisation}]
 Let $\mu$ be the absolutely continuous invariant measure on $\Omega$.
By Lemma~\ref{lemma.equality} the values of $\hat\alpha_k$ are given by the expression
in the statement of the theorem. The parameters $t\lambda_k$ are then given by
Theorem~\ref{theorem.lambda} since the assumption $\sum_kk\hat\alpha_k<\infty$
is satisfied by uniform expansiveness which implies that $\hat\alpha_k$ decay
exponentially fast.

 In order to apply Theorem~\ref{thm1} it remains to verify
Assumptions (I)--(VI).
 Assumption~(IV) is satisfied for any $d_0<1<d_1$ arbitrarily close to $n-1$.
Since the unstable manifold is all of $\Omega$, Assumption~(V) is satisfied for
any $u_0<n-1$ arbitrarily close to $n-1$. Similarly, Assumption~(VI) is satisfied with
$\beta=\eta=1$. Assumption~(III) is satisfied as $T$ is uniformly expanding (III-i) is
trivially satisfied with $q=\infty$. (III-ii) follows from the regularity of the map and
(III-iii) is satisfied since the contraction is in fact exponential.
Assumption~(II) is satisfied by a result of Saussol~\cite{Sau00} where the the decay of correlations is
shown for functions of bounded variation vs $\mathscr{L}^1$. Since characteristic functions have bounded variation
we can take $\phi=\tilde\phi=\ind_{U}=\ind_{B_\rho(\Gamma)}$ in Section~\ref{est_R1_section}
 and since functions that are bounded in the supremum norm (as characteristic functions are)
are automatically in $\mathscr{L}^1$ the assumption is fulfilled.
\end{proof}

\noindent In the special case when the coupling constant $\gamma$
is equal to zero, then $\mu$ is the product measure of the absolutely continuous
$T$-invariant measure $\hat\mu$ on the interval $I=[0,1]$, that is
$\mu=\hat\mu\times\hat\mu\times\cdots\times\hat\mu$, $n$ times.
 Consequently the
density $h$ on the diagonal $\Delta$ is equal to $\hat{h}^n$, where
$\hat{h}=\frac{d\hat\mu}{dx}$. Then we conclude as follows:

\begin{cor} Let $\Omega=I^n$ and $T:\Omega\circlearrowleft$ be the $n$-fold product of
a uniformly expanding map $T: I\circlearrowleft$ with a.c.i.m\ $\hat\mu$ with
density $\hat{h}$. Then
$$
 \hat{\alpha}_{k+1}
 =\frac1{ \int_I\hat{h}^n(x) \,dx}\int_I \frac{\hat{h}^{n}(x)}{|DT^k(x)|^{n-1}}\,dx
 $$
 and in particular
 \begin{eqnarray*}
\lambda_{k}&=&\frac1{\alpha_1}(\hat\alpha_k-2\hat\alpha_{k+1}+ \hat{\alpha}_{k+2})\\
 &=&\frac{\alpha_1^{-1}}{ \int_I\hat{h}^n(x) \,dx}\int_I\hat{h}^{n}(x)\!
 \left(\frac1{|DT^k(x)|^{n-1}}-\frac2{|DT^{k+1}(x)|^{n-1}}+\frac1{|DT^{k+2}(x)|^{n-1}}\right)\,dx,
 \end{eqnarray*}
 where
 $$
 \alpha_1=1-\hat\alpha_2
 =\left( \int_I\hat{h}^n(x) \,dx\right)^{-1}\int_I\hat{h}^n(x)\!\left(1- |DT(x)|^{-(n-1)}\right)\,dx
 $$
 is the extremal index.
 \end{cor}

 \noindent For $n=2$ these formulas were derived by Coelho and Collet~\cite{CC94} Theorem~1.

\section{Aknowledgements}
\noindent SV thanks the {\em Laboratoire International Associ\'e LIA LYSM}, the INdAM (Italy) and the UMI-CNRS 3483, {\em Laboratoire  Fibonacci} (Pisa) where this work has been completed under a CNRS delegation. \\
NH thanks the University of Toulon and the Simons Foundation (award ID 526571) for support.


\end{document}